\documentclass[11pt, onesided, leqno]{amsart}
\usepackage{enumerate}
\usepackage[dvips]{graphics}
\usepackage{amsmath,amssymb,mathrsfs}
\newcommand{\doublespace}
   {\addtolength{\baselineskip}{0.15\baselineskip}}

\newtheorem{pdef}{Definition}[section] %
\newtheorem{thm}[pdef]{Theorem}        % begin{thm},{cor},{lem},{remark}
\newtheorem{cor}[pdef]{Corollary}
\newtheorem{lem}[pdef]{Lemma}

\newtheorem{exam}[pdef]{Example}
\newtheorem{remark}[pdef]{Remark}
\newtheorem{prop}[pdef]{Proposition}

\newcounter{equationnumber}
\renewcommand{\theequation}{\thesection.\arabic{equation}}
\def\mathletters{
    \addtocounter{equation}{1}
    \edef\@currentlabel{\theequation}
    \setcounter{equationnumber}{\value{equation}}
    \setcounter{equation}{0}
    \edef\theequation{\@currentlabel\noexpand\alph{equation}}
    }

\title{Regularity results for free L\'{e}vy processes}
\author{Hao-Wei Huang and Jiun-Chau Wang}
\date {\today}

\address{Department of Mathematics, National Tsing Hua University,
No. 101, Section 2, Kuang-Fu Road, Hsinchu 300044, Taiwan R.O.C.}

\email{huanghw@math.nthu.edu.tw}

\address{Department of Mathematics and Statistics, University of
Saskatchewan, Saskatoon, Saskatchewan S7N 5E6, Canada}
\email{jcwang@math.usask.ca}

\begin{document}

\begin{abstract} Given a free additive convolution semigroup $\left(\mu_t\right)_{t\geq 0}$ and a probability measure $\nu$ on $\mathbb{R}$, we find the necessary and sufficient conditions for the process $\mu_t \boxplus \nu$ to be Lebesgue absolutely continuous with a positive and analytic density throughout $\mathbb{R}$ at all time $t>0$. For semigroups without this property, we find the necessary and sufficient conditions for the density of $\mu_t \boxplus \nu$ to be analytic at its zeros. These results are quantified by the L\'{e}vy measure of the semigroup, making it fairly easy to construct many concrete examples. Finally, we show that $\mu_t \boxplus \nu$ has a finite number of connected components in its support if both the L\'{e}vy measure of $\left(\mu_t\right)_{t \geq 0}$ and the initial law $\nu$ do.
\end{abstract}
\maketitle \doublespace \pagestyle{myheadings} \thispagestyle{plain}
\markboth{   }{ }


\section{Introduction}
Recall from Biane's work [\ref{Biane98}] that a free (additive) L\'{e}vy process $\{Z_t\}_{t\geq 0}$ is a noncommutative stochastic process of self-adjoint random variables such that $Z_{0}=0$, $Z_t$ tends weakly to $0$ as $t\rightarrow 0$, and the increments $Z_t-Z_s$ are freely independent and stationary. Thus, denoting by $\mu_t$ the distribution of $Z_t$, the family $\left(\mu_t\right)_{t\geq 0}$ satisfies the weak convergence $\mu_t \Rightarrow \mu_{0}=\delta_0$ as $t\rightarrow 0$ and forms a \emph{$\boxplus$-semigroup} under the free additive convolution $\boxplus$ in the sense that \[\mu_{t+s}=\mu_{t}\boxplus\mu_{s},\quad s,t \geq 0.\] In particular, every $\mu_t$ is \emph{$\boxplus$-infinitely divisible}; and conversely, every $\boxplus$-infinitely divisible measure $\mu$ embeds in a unique $\boxplus$-semigroup $\left(\mu_t\right)_{t\geq 0}$ such that $\mu_1=\mu$. Since its appearance in [\ref{Biane98}], there has been an intensive research on free L\'{e}vy processes. For examples, see [\ref{Nielsen05}] for a free analogue of the L\'{e}vy-It\^{o} decomposition,  [\ref{Anshelevich02}] for a construction of stochastic integrals, and [\ref{Anshelevich13}] for a formula of the infinitesimal generator of a free L\'{e}vy process.

The current paper contributes to this line of research by proving regularity results of the free convolution $\mu_t\boxplus \nu$ where $\nu$ belongs to $\mathcal{P}_{\mathbb{R}}$, the set of all probability measures on $\mathbb{R}$. Our results are concerned with the regularizing effect of $\left(\mu_t \right)_{t \geq 0}$. Recall the \emph{property} (H) for a measure $\mu\in\mathcal{P}_{\mathbb{R}}$ as follows:
\begin{enumerate}
\item [(H)] For all $\nu \in \mathcal{P}_{\mathbb{R}}$, the measure $\mu \boxplus \nu$ is Lebesgue absolutely continuous and has a strictly positive, analytic density everywhere on $\mathbb{R}$.
\end{enumerate} A process $\left(\mu_t \right)_{t \geq 0}$ is said to have property (H) if every marginal distribution $\mu_t$ has this property. Thus, such a $\left(\mu_t \right)_{t \geq 0}$ can serve as a distributional mollifier in a rather strong sense. Property (H) was introduced and studied in the paper  [\ref{BBG09}]  where some sufficient conditions for $\left(\mu_t \right)_{t \geq 0}$ to satisfy (H) were found, along with some examples. It was also shown in [\ref{BBG09}]  that in order for $\left(\mu_t\right)_{t\geq 0}$ to have property (H), each $\mu_t$ cannot have finite second moment.

For $\nu\in\mathcal{P}_{\mathbb{R}}$ and an $\boxplus$-infinitely divisible measure $\mu$, a formula for the density $p_{\mu\boxplus\nu}$ of the absolutely continuous part of $\mu\boxplus\nu$ is found in Theorem \ref{thmR1} and we show later in Theorem \ref{propertyH1} that $p_{\mu\boxplus\nu}$ is everywhere positive and analytic on $\mathbb{R}$ if and only if \[\int_{\mathbb{R}}s^2\,d\sigma_{\mu}(s)=+\infty\quad\text{and}\quad \int_{\mathbb{R}}\frac{1+s^2}{(x-s)^2}\,d\sigma_{\mu}(s)>1\] for every $x\in \mathbb{R}$. Here the measure $\sigma_{\mu}$ is the \emph{L\'{e}vy measure} of $\mu$ (see Section 2). This result yields the necessary and sufficient conditions for a free L\'{e}vy process to have property (H), as seen in Corollary \ref{propertyH3}. Notice that the value of the second moment $m_2(\sigma_{\mu})$ of $\sigma_{\mu}$ being infinite is known to be equivalent to $m_2(\mu)=+\infty$. For a semigroup $\left(\mu_t \right)_{t \geq 0}$ without (H), in Theorem \ref{analyticR1}, we find the necessary and sufficient conditions for the density $p_{\mu_t\boxplus \nu}$ to be analytic at its zeros. Examples are given to illustrate such an analyticity behavior, and our result subsumes all known cases of this phenomenon in the context of free L\'{e}vy process, including Biane's classical example of free Brownian motion [\ref{Biane97}]. Finally, Theorem \ref{nocomponentR2} shows that $\mu\boxplus \nu$ has a finite number of connected components in its support if both the L\'{e}vy measure $\sigma_{\mu}$ and $\nu$ do.

The main tool in our proofs is a global inversion result which can be traced back to Proposition 5.12 in the original free convolution paper [\ref{BerVoicu93}] of Bercovici and Voiculescu. Such a global inversion technique is the basis of Biane's analysis for free Brownian motion in [\ref{Biane97}]. Later on, this technique was further studied in great detail and applied to the study of partially defined free convolution semigroups in [\ref{BelBer05}]. While the arguments in [\ref{BelBer05}] made a clever use of the Denjoy-Wolff fixed point theory to produce a powerful general theory of global inversion, our approach based on calculus and integration is tailor-made for free L\'{e}vy processes. This real-variable method is closely related to Biane's work [\ref{Biane97}] and the analysis used in [\ref{Huang14}] for partially defined free convolution semigroups. In addition, our method makes it fairly easy to construct processes with specific regularity properties, as illustrated by the examples in this paper.

The multiplicative analogs of regularities properties will appear in a forthcoming paper.

This paper is organized as follows. After collecting some preliminary material in Section 2, we state our main results and examples in Section 3. The detailed proofs of these results are presented in Section 4. We shall now begin by reviewing the essence of free harmonic analysis.

\section{Preliminaries}
\subsection{Integral transforms} Denote by $\mathcal{M}_{\mathbb{R}}$ the family of finite positive Borel measures on $\mathbb{R}$ and recall that $\mathcal{P}_{\mathbb{R}}=\{\mu\in \mathcal{M}_{\mathbb{R}}: \mu(\mathbb{R})=1\}$. For any $\nu \in \mathcal{M}_{\mathbb{R}}\setminus \{0\}$, we consider its \emph{Cauchy transform} \[G_{\nu}(z)=\int_{\mathbb{R}} \frac{1}{z-s}\,d\nu(s),\] an analytic map from the complex upper half-plane $\mathbb{C}^{+}$ into the lower half-plane $\mathbb{C}^{-}=-\mathbb{C}^{+}$. The map $G_{\nu}$ determines the measure $\nu$ uniquely. The \emph{$F$-transform} of $\nu$ is defined to be the reciprocal $F_{\nu}=1/G_{\nu}$, an analytic self-map of $\mathbb{C}^{+}$.

It is well-known that every analytic map $F:\mathbb{C}^{+}\rightarrow \mathbb{C}^{+}\cup\mathbb{R}$ admits a unique \emph{Nevanlinna representation}: \begin{equation}\label{eq:2.2}
F(z)=a+bz+\int_{\mathbb{R}}\frac{1+sz}{s-z}\,d\rho(s), \quad z \in \mathbb{C}^{+},
\end{equation} where $a=\Re F(i)$, $b=\lim_{y\rightarrow \infty}F(iy)/iy\geq 0$, and $\rho \in \mathcal{M}_{\mathbb{R}}$. The measure $\rho$ in this case will be called the \emph{$F$-representing measure} of $F$. Thus, every $F$-transform $F_\nu$ admits the integral form \eqref{eq:2.2}, with $b=1/\nu(\mathbb{R})$, and we write $\rho=\rho_\nu$ in this case. By virtue of the identity $sz/(s-z)=z+z^2/(s-z)$, we sometimes write the integral form \eqref{eq:2.2} in the following way: \[F(z)=a+(b+\rho(\mathbb{R}))z-(1+z^2)G_{\rho}(z),\quad z\in\mathbb{C}^{+}.\]

It follows from \eqref{eq:2.2} that the $F$-transform of a measure $\nu\in \mathcal{P}_{\mathbb{R}}$ satisfies $\Im F_{\nu}(z)\geq \Im z$ for $z \in \mathbb{C}^{+}$. If equality occurs for some $z$ in $\mathbb{C}^+$, then $\nu$ is \emph{degenerate}, that is, $\nu$ is the point mass $\delta_a$ at some $a\in\mathbb{R}$. Moreover, the limit $\lim_{y\rightarrow \infty} F_{\nu}(iy)/iy =1$ implies that there exist positive numbers $\alpha$ and $\beta$ such that $F_{\nu}$ has an analytic right inverse $F_{\nu}^{-1}$ defined in the truncated cone $\Gamma_{\alpha,\beta}=\{x+iy: |x|<\alpha y,\,y>\beta\}$. The \emph{Voiculescu transform} $\varphi_{\nu}$ is defined as $\varphi_{\nu}(z)=F_{\nu}^{-1}(z)-z$, and the identity \[\varphi_{\nu_1 \boxplus \nu_2}(z)=\varphi_{\nu_1}(z)+\varphi_{\nu_2}(z)\] holds in a cone on which the three Voiculescu transforms are defined, see [\ref{BerVoicu93}]. The Voiculescu transform $\varphi_{\nu}$ is related to the \emph{$R$-transform} $R_{\nu}$ via the formula $R_{\nu}(z)=\varphi_{\nu}(1/z)$.

\subsection{Infinite divisibility and L\'{e}vy measure} Recall that a measure $\mu \in \mathcal{P}_{\mathbb{R}}$ is said to be \emph{$\boxplus$-infinitely divisible} if for each $n \in \mathbb{N}$, there exists a measure $\mu_{n} \in \mathcal{P}_{\mathbb{R}}$ such that $\mu=\mu_{n}\boxplus\cdots \boxplus \mu_{n}$ ($n$ times). The set of all $\boxplus$-infinitely divisible measures is denoted by $\mathcal{ID}(\boxplus)$. It was shown in [\ref{BerVoicu93}] that $\mu\in \mathcal{ID}(\boxplus)$ if and only if its Voiculescu transform $\varphi_{\mu}$ satisfies the \emph{free L\'{e}vy-Hin\v{c}in formula}:
\begin{equation} \label{eq:2.3}
\varphi_{\mu}(z)=\gamma+\int_{\mathbb{R}}\frac{1+sz}{z-s}\,d\sigma(s), \quad z \in \mathbb{C}^{+},
\end{equation} where $\gamma \in \mathbb{R}$ and $\sigma \in \mathcal{M}_{\mathbb{R}}$. The number $\gamma$ and the measure $\sigma$ in \eqref{eq:2.3} are unique, and the measure $\sigma$ will be called the \emph{L\'{e}vy measure} of the infinitely divisible measure $\mu$. We  write $\gamma=\gamma_{\mu}$ and $\sigma=\sigma_{\mu}$ to indicate the correspondence between these parameters and the measure $\mu$ through \eqref{eq:2.3}.

The equivalence of classical, free, and Boolean central limit theorems [\ref{Pata}, \ref{BerPata99}] shows that $m_2(\mu)<+\infty$ if and only if $m_2(\sigma_{\mu})<+\infty$ if and only if $m_2(\rho_{\mu})<+\infty$, where $\rho_{\mu}$ is the $F$-representing measure of $F_{\mu}$; in which case we also have the first moment $m_1(\mu)=\gamma_{\mu}+m_1(\sigma_{\mu})$ and the variance $var(\mu)=m_2(\mu)-[m_1(\mu)]^2=\sigma_{\mu}(\mathbb{R})+m_2(\sigma_{\mu})$.

Note that $\varphi_{\mu_t}(z)=t\,\varphi_{\mu_1}(z)$ for all $z\in\mathbb{C}^{+}$ and $t>0$ in a $\boxplus$-semigroup $\left(\mu_t\right)_{t\geq 0}$. In particular, one has $\sigma_{\mu_t}=t\,\sigma_{\mu_1}$ and we shall call $\sigma_{\mu_1}$ the L\'{e}vy measure of the semigroup $\left(\mu_t\right)_{t\geq 0}$.

\subsection{Boundary behavior of analytic maps} Given a nonconstant map $f:\mathbb{C}^{+}\rightarrow \mathbb{C}$, its \emph{vertical limit} $f^{*}(\alpha)$ at a point $\alpha\in\mathbb{R}$ is defined as \[f^{*}(\alpha)=\lim_{\epsilon \rightarrow 0^{+}}f(\alpha+i\epsilon),\] and $f$ is said to have an \emph{angular derivative} $f^{\prime}(\alpha)$ at $\alpha$ if for some $s\in \mathbb{R}$, the limit \[f^{\prime}(\alpha)=\lim_{z\rightarrow_{\sphericalangle}\alpha} \frac{f(z)-s}{z- \alpha}\] exists in $\mathbb{C}$. Here the \emph{non-tangential convergence} $z\rightarrow_{\sphericalangle}\alpha$ means that the quantity $|\Re z -\alpha|/\Im z$ remains bounded as $z$ tends to the real number $\alpha$. We also say that $z\rightarrow \infty$ non-tangentially if $|z|\rightarrow +\infty$ and $|\Re z|/\Im z$ is bounded. Note that if $f^{\prime}(\alpha)$ exists then $\lim_{z\rightarrow_{\sphericalangle}\alpha}f(z)$ exists and is equal to the real number $s$.

We mention a few well-known results about non-tangential and vertical limits. The upper half-plane version presented here is summarized from Section 2 of [\ref{Belinschi08}]. First, Fatou Theorem shows that if $f:\mathbb{C}^{+}\rightarrow \mathbb{C}^{+}\cup\mathbb{R}$ is analytic, then $f^{*}(\alpha)$ exists in $\mathbb{C}$ for almost all $\alpha \in \mathbb{R}$ with respect to Lebesgue measure $\lambda$. Next, Lindel\"{o}f theorem for an analytic map $f:\mathbb{C}^{+}\rightarrow \mathbb{C}$ states that if $\mathbb{C}\setminus f(\mathbb{C}^{+})$ contains at least three points and if $\gamma:[0,1) \rightarrow \mathbb{C}^{+}$ is a curve such that $\lim_{t\rightarrow 1^{-}}\gamma(t)=\alpha \in \mathbb{R}$ and the limit $L=\lim_{t\rightarrow 1^{-}}f(\gamma(t))$ exists in the \emph{extended complex plane} $\mathbb{C}\cup \{\infty\}$, then the non-tangential limit $\lim_{z\rightarrow_{\sphericalangle}\alpha} f(z)$ exists in $\mathbb{C}\cup \{\infty\}$ and $\lim_{z\rightarrow_{\sphericalangle}\alpha} f(z)=L$. In particular, we have \[\lim_{z\rightarrow_{\sphericalangle}\alpha}  f(z)=f^{*}(\alpha),\] provided that any of the two limits exists in $ \mathbb{C}\cup \{\infty\}$.

The Julia-Carath\'{e}odory theory for the angular derivatives is as follows (cf. Theorem 2.1 of [\ref{BelinschiJCW}]). Suppose that $f$ is an analytic self-map of $\mathbb{C}^{+}$ and $\alpha\in \mathbb{R}$. One has
\[
c=\liminf_{z\rightarrow\alpha}\frac{\Im f(z)}{\Im z}<+\infty
\]
if and only if the angular derivative $f^{\prime}(\alpha)$ exists; in which case we have $c>0$, $f^*(\alpha)\in\mathbb{R}$, and
\begin{equation}\label{eq:2.7}
f^{\prime}(\alpha)=c.
\end{equation} Alternatively, assuming $f^*(\alpha)\in\mathbb{R}$, one has $c=+\infty$ if and only if
\[
\lim_{z\rightarrow_{\sphericalangle}\alpha}\left|\frac{f(z)-f^*(\alpha)}{z-\alpha}\right|=+\infty,
\] and we write $f^{\prime}(\alpha)=+\infty$ in this case.

The boundary behavior of $F$-transform is closely related to the regularity of the underlying measure. We now review the relevant results from the books [\ref{Cauchytrans}] and [\ref{BSimon}].
First, recall that a \emph{support} of a Borel measure $\mu$ is a Borel set whose complement is a $\mu$-measure zero set. The \emph{topological support} $\text{supp}(\mu)$ is the intersection of all closed sets that support the measure $\mu$.

Let $\mu=\mu_{\text{ac}}+\mu_{\text{s}}$ be the Lebesgue decomposition of $\mu \in \mathcal{P}_{\mathbb{R}}$. The singular part $\mu_{\text{s}}$ is supported on the $\lambda$-measure zero set \[\{s \in \mathbb{R}: (-\Im G_{\mu})^{*}(s)=+\infty\},\] and the atoms of $\mu$ can be detected through the non-tangential limit [\ref{BerVoicu98}]: \begin{equation}\label{eq:2.8}\lim_{z\rightarrow_{\sphericalangle}s} (z-s)G_{\mu}(z)=\mu(\{s\}),\quad s \in \mathbb{R}.\end{equation} So, a point $\alpha\in\mathbb{R}$ is an atom of $\mu$ if and only if $F_{\mu}^{*}(\alpha)=0$ and the angular derivative $F_{\mu}^{\prime}(\alpha)=1/\mu(\{\alpha\})<+\infty$. The absolutely continuous part $\mu_{\text{ac}}$ may be given by \[d\mu_{\text{ac}}(s)=\frac{1}{\pi} (-\Im G_{\mu})^{*}(s)\,d\lambda(s),\] where the vertical limit function $(-\Im G_{\mu})^{*}$ is well-defined almost everywhere on $\mathbb{R}$ by the aforementioned result of Fatou.

We now present some useful consequences of the Julia-Carath\'{e}odory theory. Most of the results below are known already. We provide their proofs here only for reader's convenience and for the sake of completeness.

\begin{prop} \label{prop2.1} Let $\nu$ be a non-zero Borel positive measure on $\mathbb{R}$ such that
\[d\rho(s)=\frac{d\nu(s)}{1+s^2} \in \mathcal{M}_{\mathbb{R}}.\]
\begin{enumerate} [$\qquad(1)$]
\item For every $\alpha \in\mathbb{R}$, we have the formula \begin{equation}\label{eq:2.9}\liminf_{z\rightarrow
\alpha}\int_\mathbb{R}\frac{d\nu(s)}{|z-s|^2}=
\int_\mathbb{R}\frac{d\nu(s)}{(\alpha-s)^2}=
\lim_{z\rightarrow_{\sphericalangle}\alpha}\int_\mathbb{R}\frac{d\nu(s)}{|z-s|^2},\end{equation} where these
equalities are considered in $(0,+\infty]$.
\item Suppose in addition that $\nu(\mathbb{R})<\infty$, and let $G_{\nu}$ be its Cauchy transform. Then the common quantity in
\eqref{eq:2.9} is finite at $\alpha\in\mathbb{R}$ if and only if the angular derivative $G_{\nu}^{\prime}(\alpha)$ exists at $\alpha$. In this case we have\emph{:} \[G_{\nu}^{\prime}(\alpha)=-\int_\mathbb{R}\frac{d\nu(s)}{(\alpha-s)^2}\] and \[G_{\nu}(z)\rightarrow_{\sphericalangle}G_{\nu}^{*}(\alpha)=\int_\mathbb{R}\frac{d\nu(s)}{\alpha-s} \quad \text{as} \quad z\rightarrow_{\sphericalangle}\alpha.\]
\item Let $F(z)=a+(b+\rho(\mathbb{R}))z-(1+z^2)G_{\rho}(z)$ be a Nevanlinna form, and recall that $d\nu(s)=(1+s^2)\,d\rho(s)$. Then \eqref{eq:2.9} has a finite value at
$\alpha\in\mathbb{R}$ if and only if the angular derivative $F^{\prime}(\alpha)$ exists at $\alpha$. In this case we have: \[F(z)\rightarrow_{\sphericalangle}F^{*}(\alpha)=a+b\alpha+\int_\mathbb{R}\frac{1+s\alpha}{s-\alpha}\,d\rho(s) \quad \text{as} \quad z\rightarrow_{\sphericalangle}\alpha,\] and \[F^{\prime}(\alpha)=b+\int_\mathbb{R}\frac{1+s^2}{(s-\alpha)^2}\,d\rho(s).\]
\end{enumerate}
\end{prop}

\begin{proof} By replacing the measure $d\nu(s)$ with $d\nu(s+\alpha)$, we may  and do assume $\alpha=0$ throughout the proof. We first prove \eqref{eq:2.9}. Fatou's Lemma and the monotone convergence theorem show that \[c=\int_\mathbb{R}\frac{d\nu(s)}{s^2}\leq
\liminf_{z\to0}\int_\mathbb{R}\frac{d\nu(s)}{|z-s|^2} \leq \lim_{\epsilon \rightarrow 0^{+}}\int_\mathbb{R}\frac{d\nu(s)}{|i\epsilon-s|^2}=\int_\mathbb{R}\frac{d\nu(s)}{s^2},\] from which we obtain the first equality in \eqref{eq:2.9}. The second equality in \eqref{eq:2.9} holds if $c=+\infty$, because \[\liminf_{z\to0}\int_\mathbb{R}\frac{d\nu(s)}{|z-s|^2} \leq \lim_{z\rightarrow_{\sphericalangle}0}\int_\mathbb{R}\frac{d\nu(s)}{|z-s|^2}.\] If $c<\infty$, then for any
$s\in\mathbb{R}\setminus \{0\}$ and $z=x+iy\in\mathbb{C}^+$ with
$|x|\leq \delta y$ for some $\delta>0$, we have
\begin{align*}
\left|\frac{1}{|z-s|^2}-\frac{1}{s^2}\right|&
=\frac{1}{s^2}\frac{|2(s-x)x+x^2-y^2|}{(x-s)^2+y^2} \\
&\leq\frac{1}{s^2}\left(\frac{2\delta|x-s|y}{(x-s)^2+y^2}+
\frac{x^2+y^2}{y^2}\right)\leq\frac{1}{s^2}\left(\delta+\delta^2+1\right)\in L^{1}(\nu).
\end{align*} The dominated convergence theorem then yields
\[\lim_{z\rightarrow_{\sphericalangle}0}\int_\mathbb{R}\frac{d\nu(s)}{|z-s|^2}=
\int_\mathbb{R}\frac{d\nu(s)}{s^2},\] finishing the proof of \eqref{eq:2.9}.

Next, we prove (2). Let $c\in(0,+\infty]$ denote the common value of \eqref{eq:2.9}. In the case of $c<+\infty$, the Cauchy-Schwarz inequality implies that the map $s \mapsto1/s$ belongs to $L^{1}(\nu)$. Then the identity
\[G_\nu(i\epsilon)+\int_\mathbb{R}\frac{d\nu(s)}{s}=\int_\mathbb{R}\frac{i\epsilon}{(i\epsilon-s)s}\,d\nu(s)\] and the dominated convergence theorem yield
\[G_\nu^{*}(0)=-\int_\mathbb{R}\frac{d\nu(s)}{s}.\] Since
\[\frac{-\Im G_{\nu}(z)}{\Im z}=
\int_\mathbb{R}\frac{d\nu(s)}{|z-s|^2},\] the assumption $c<+\infty$ and \eqref{eq:2.7} imply $-G_\nu^{\prime}(0)=c$. Moreover, as $z\rightarrow_{\sphericalangle}0$, say, $|\Re z|\leq \delta \Im z$, we have \begin{align*}
\left|\frac{\Re G_{\nu}(z)-G_{\nu}^{*}(0)}{\Im G_{\nu}(z)}\right|&
=\frac{|z|}{\Im z}\left|\frac{\Re G_{\nu}(z)-G_{\nu}^{*}(0)}{z}\right|\frac{\Im z}{-\Im G_{\nu}(z)}\\
&\leq \sqrt{1+\delta^2}\left|\frac{G_{\nu}(z)-G_{\nu}^{*}(0)}{z}\right|\frac{2}{c}\\& \rightarrow 2\sqrt{1+\delta^2}.
\end{align*} Therefore, the convergence $G_\nu(z)\rightarrow G_{\nu}^{*}(0)$ must be non-tangential as $z\rightarrow_{\sphericalangle}0$. Conversely, if $G_\nu^{\prime}(0)$ exists, then we have $c=-G_\nu^{\prime}(0)< +\infty$ by \eqref{eq:2.7}. As seen earlier, the integral formula of $G_\nu^{*}(0)$ follows immediately. The proof of (2) is completed.

The proof of (3) follows from the same arguments as in the proof of (2). We omit the details.
\end{proof}

The attentive reader will notice that the non-tangential convergence results in the preceding proposition can be strengthened as follows. The proofs are left to the reader.
\begin{prop} Given a map $f:\mathbb{C}^+\rightarrow \mathbb{C}\cup\{\infty\}$, let $\{\gamma(t):t\in[0,1)\}\subset \mathbb{C}^+$ be a curve such that $\gamma(t)\rightarrow_{\sphericalangle}\alpha\in \mathbb{R}$ and $f(\gamma(t))\subset \mathbb{C}^+$ as $t\rightarrow 1^{-}$. If \[\lim_{t\rightarrow 1^{-}}\frac{f(\gamma(t))-L}{\gamma(t)-\alpha}> 0\] for some $L\in \mathbb{R}$, then $f(\gamma(t))\rightarrow_{\sphericalangle}L$ as $t\rightarrow 1^{-}$.
\end{prop}

The next result describes the atoms of a measure.

\begin{cor} \label{measureatom}
Let $\nu\in \mathcal{P}_{\mathbb{R}}$ be a nondegenerate measure whose $F$-representing measure is $\rho_{\nu}$. Given $\alpha \in \mathbb{R}$, we define \[I_{1}=\int_\mathbb{R}\frac{1+s^2}{(\alpha-s)^2}\,d\rho_{\nu}(s)\quad \text{and} \quad I_{2}=\int_\mathbb{R}\frac{d\nu(s)}{(\alpha-s)^2}.\]
\begin{enumerate}[$\qquad(a)$]
\item {$\nu(\{\alpha\})>0$ if and only if $F_{\nu}^{*}(\alpha)=0$ and $I_1<+\infty$,
and in this case we have $(1+I_1)\nu(\{\alpha\})=1$.}
\item {$\rho_{\nu}(\{\alpha\})>0$ if and only if $G_{\nu}^{*}(\alpha)=0$ and $I_2<+\infty$,
in which case we have $I_2(1+\alpha^2)\rho_{\nu}(\{\alpha\})=1$.}
\end{enumerate}
\end{cor}

\begin{proof} Part (a) follows directly from \eqref{eq:2.8} and Proposition \ref{prop2.1} (3). We show (b). First, we write the Nevanlinna form \[F_{\nu}(z)=\Re F_{\nu}(i)+(1+\rho_{\nu}(\mathbb{R}))z-(1+z^2)G_{\rho_{\nu}}(z),\quad z \in \mathbb{C}^{+}.\] If $\rho_{\nu}(\{\alpha\})>0$, then \eqref{eq:2.8} shows that \[\lim_{z\rightarrow_{\sphericalangle}\alpha} \frac{G_{\nu}(z)}{z- \alpha}=\frac{-1}{(1+\alpha^2)\rho_{\nu}(\{\alpha\})}.\] Hence, the derivative $G_{\nu}^{\prime}(\alpha)$ exists and $G_{\nu}^{*}(\alpha)=0$. Proposition \ref{prop2.1} (2) then shows that $I_2 < +\infty$ and $(1+\alpha^2)\rho_{\nu}(\{\alpha\})=1/I_2$. The converse follows from the same consideration.
\end{proof}

\subsection{Analytic continuation} Throughout the paper, if a map $f$ defined on $\mathbb{C}^{+}$ can be extended continuously or analytically to points on $\mathbb{R}$, we will use the same notation $f$ for the extension. The following result of Greenstein [\ref{Greenstein60}] gives the conditions under which an analytic map on $\mathbb{C}^{+}$ can be continued analytically from $\mathbb{C}^{+}$ into a subset of $\mathbb{C}^{-}$ through a finite interval on $\mathbb{R}$.

\begin{prop} [Greenstein] \label{prop2.4} Let $a<b$ be two real numbers.
\begin{enumerate}
\item[\emph{(1)}] Let $\rho$ be the $F$-representing measure of an analytic map
$F:\mathbb{C}^{+}\rightarrow \mathbb{C}^{+}\cup \mathbb{R}$. The map $F$ extends analytically across the interval $(a,b)$ into $\mathbb{C}^{-}$ if and only if the restriction of $\rho$ on $(a,b)$ is absolutely continuous and has a real-analytic density $p$ on $(a,b)$. In this case the continuation of $F$ into some domain $U\subset\mathbb{C}^{-}\cup(a,b)$ is given by \[F(z)=\overline{F(\overline{z})}+2\pi i(1+z^2)p(z),\quad z\in U,\] where $p(z)$ denotes the complex-analytic extension of the real-analytic density $p$. In particular, the map $F$ continues analytically across $(a,b)$ into $\mathbb{C}^{-}$ by reflection $F(z)=\overline{F(\overline{z})}$, $z\in \mathbb{C}^{-}\cup(a,b)$, if and only if $\rho((a,b))=0$.
\item[\emph{(2)}] The Cauchy transform $G_{\nu}$ of a measure $\nu\in \mathcal{M}_{\mathbb{R}}$
extends analytically through $(a,b)$ into $\mathbb{C}^{-}$ if and only if the restriction of $\nu$ on $(a,b)$ is absolutely continuous with a real-analytic density $p_{\nu}$ on $(a,b)$. The extension of $G_{\nu}$ in this case is given by $G_{\nu}(z)=\overline{G_{\nu}(\overline{z})}-2\pi ip_{\nu}(z)$ for $z$ in some domain $U\subset\mathbb{C}^{-}\cup(a,b)$, where $p_{\nu}(z)$ is the analytic continuation of $p_{\nu}$.
\end{enumerate}
\end{prop}

In the next result, an atom $s$ of a measure $\mu$ is said to be \emph{isolated} if there exists an open interval $I$ such that $\mu$ is atomless on $I\setminus\{s\}$.

\begin{cor} \label{analyextR2} Let $\mu\in \mathcal{M}_{\mathbb{R}}$ and $s\in\mathbb{R}$ be so that
the singular part of the restriction of $\mu$ to some open interval containing $s$ consists of at most one atom. Then the continuous part $d\mu_{\text{ac}}/d\lambda$ of $\mu$ is \emph{(}real\emph{)} analytic at $s$ if and only if one and only one of the following situations occurs:
\begin{enumerate}
\item[\emph{(1)}] $G_{\mu}$ extends analytically to an open disk centered at $s$;
\item[\emph{(2)}] $G_{\mu}$ extends analytically to a punctured disk centered at $s$ and the point $s$ is a simple pole
of this extension.
\end{enumerate} In \emph{(2)}, the point $s$ is an isolated atom of $\mu$ and $G_\mu$ has the residue $\mu(\{s\})$ at $s$.
\end{cor}

\begin{proof} Let $p=\mu(\{s\})$. The case of $p=0$ is the result of Greenstein. We shall assume $p>0$.

Suppose first that the density $d\mu_{\text{ac}}/d\lambda$ is analytic at $s$. Observe that the resection of
the finite measure $\nu=\mu-p\delta_{s}$ on some open interval containing $s$
is absolutely continuous with the density $d\nu/d\lambda=d\mu_{\text{ac}}/d\lambda$. So, by Greenstein's result, the function \[G_{\nu}(z)=G_{\mu}(z)-\frac{p}{z-s}\] extends analytically to a neighborhood of $s$. Hence, (2) occurs, and the residue of $G_{\mu}$ at $s$ is $p$ in this case.

Conversely, assume that $G_{\mu}$ extends meromorphically to $s$ as stated in (2). Then the principal part of the Laurent series of $G_{\mu}$ around $s$ is $p/(z-s)$, and the regular part $G_{\mu}(z)-p/(z-s)$ coincides with the Cauchy transform of
$\nu=\mu-p\delta_{s}$ in a neighborhood of $s$. Consequently, $G_{\nu}$ extends analytically to $s$, implying that the density $d\mu_{\text{ac}}/d\lambda=d\nu/d\lambda$ is analytic at $s$ by Proposition \ref{prop2.4}. \end{proof}

We note a useful observation on the Nevanlinna forms for which $b=0$.

\begin{cor} \label{cor2.6} Let $F(z)=a+\rho(\mathbb{R})z-(1+z^2)G_{\rho}(z)$ be a Nevanlinna form. The restriction of the measure $\rho$ to $\{s\in\mathbb{R}: |s|>R\}$ is absolutely continuous with an analytic density for some $R>0$ if and only if $F(1/z)$, $z\in \mathbb{C}^{-}$, extends analytically to a neighborhood of zero.
\end{cor}
\begin{proof} Since \[F(1/z)=a-\rho(\{0\})z+\int_{\mathbb{R} \setminus \{0\}}\frac{1+sz}{z-s}\,d\rho(1/s),\quad z \in \mathbb{C}^{-},\] the result follows directly from Proposition \ref{prop2.4}.
\end{proof}

In view of the preceding results, we introduce the following definition.

\begin{pdef} \label{def2.7}
\emph{A measure $\rho\in \mathcal{M}_{\mathbb{R}}$ is
\emph{analytic} at $s\in\mathbb{R}$ if there exists an open
interval $I$ containing $s$ such that the restriction of $\rho$ to $I$ is Lebesgue absolutely continuous and admits an analytic density on $I$. The measure
$\rho$ is analytic at the point $\infty$ if there exists
a number $R>0$ so that the restriction of the measure $\rho$ to
$\mathbb{R}\setminus[-R,R]$ is absolutely continuous with an analytic density. Finally, $\rho$ is said to be
\emph{meromorphic} at $s\in\mathbb{R}$ if there exists an open
interval $I$ containing $s$ so that the restriction of $\rho$ to $I$ has no singular continuous part, $s$ is the only atom for this restriction, and the
density of this restriction is analytic at $s$.} \qed
\end{pdef} Note that all compactly supported measures are analytic at $\infty$.

\setcounter{equation}{0}
\section{Main Results and Examples}
\subsection{Global inversion and applications} Fix $a\in\mathbb{R}$, $b>0$, and a measure $\rho \in \mathcal{M}_{\mathbb{R}}\setminus\{0\}$. We consider the analytic function \begin{equation} \label{eq:3.1}
H(z)=a+bz+\int_\mathbb{R}\frac{1+sz}{z-s}\,d\rho(s),\quad z\in\mathbb{C}^+, \end{equation} which satisfies the inequality \[\Im H(x+iy)=y \left[b-\int_{\mathbb{R}}\frac{1+s^2}{(x-s)^2+y^2}\,d\rho(s)\right] \leq by, \quad x\in\mathbb{R},\,\,y>0,\] and $\Im H(z)/\Im z\rightarrow b$ as $z\rightarrow_{\sphericalangle}\infty$. The latter limit shows that the open pre-image \[\Omega=\{z\in\mathbb{C}^+:\Im H(z)>0\}\] contains $iy$ for sufficiently large $y>0$, and that the nonnegative function
\[f(x)=\inf\left\{y>0:\int_\mathbb{R}\frac{1+s^2}{(x-s)^2+y^2}\;d\rho(s)<b\right\},\quad x\in\mathbb{R},\] has a finite value at every $x\in\mathbb{R}$. We introduce the sets \[V=\{x\in\mathbb{R}:f(x)>0\}\] and $\mathbb{R}\setminus V=\{x\in\mathbb{R}:f(x)=0\}$. Recall that the singular integral transform \[g(x)=\int_\mathbb{R}\frac{1+s^2}{(x-s)^2}\,d\rho(s), \quad x\in\mathbb{R},\] takes values in $(0,+\infty]$.
Items (3) and (4) in the following result have been discovered and studied in [\ref{BelBer05}] using the Denjoy-Wolff fixed point theory.
\begin{prop}\emph{(Global Inversion).}\label{Hprop1}
\begin{enumerate} [$\qquad(1)$]
\item {The function $f:\mathbb{R}\rightarrow [0,+\infty)$ is continuous,
and the map $g:\mathbb{R}\rightarrow (0,+\infty]$ is lower semi-continuous. We have
\begin{equation} \label{Vdef}
V=\{x\in\mathbb{R}:g(x)>b\}.
\end{equation}}
\item {The open set $\Omega$ is a simply connected domain in $\mathbb{C}^{+}\cup\mathbb{R}$,
whose topological boundary $\partial\Omega$ is the graph of the function $f$, that is, \[\partial\Omega=\{x+if(x):x\in\mathbb{R}\}.\] Also, we have the characterization
\[\Omega=\{x+iy\in\mathbb{C}^+:y>f(x)\}.\]}
\item {The function $H:\Omega\to\mathbb{C}^{+}$ is an analytic bijective map,
and it extends continuously to the topological closure $\overline{\Omega}$. The extension of $H$ satisfies
\[|H(z_1)-H(z_2)|\leq2b|z_1-z_2|,\;\;\;\;\;z_1,z_2\in\overline{\Omega}.\]}
\item {There exists a continuous function
$\omega:\mathbb{C}^+\cup\mathbb{R}\to\overline{\Omega}$ such that
$\omega:\mathbb{C}^+\rightarrow \Omega$ is an analytic bijective map,
$\omega(\mathbb{R})=\partial\Omega$, $H(\omega(z))=z$ for
$z\in\mathbb{C}^+\cup\mathbb{R}$, and
$\omega(H(z))=z$ for $z\in\overline{\Omega}$. In addition, one has
\[|\omega(z_1)-\omega(z_2)|\geq\frac{|z_1-z_2|}{2b},\;\;\;\;\;z_1,z_2\in\mathbb{C}^+\cup\mathbb{R}.\]}
\item {The function $h$ defined by
$h(x)=H(x+if(x))$ for $x\in\mathbb{R}$ is a homeomorphism from
$\mathbb{R}$ onto $\mathbb{R}$. The inverse $h^{-1}$ of $h$ is given by  $h^{-1}(s)=\Re \omega (s)$, $s\in\mathbb{R}$.} Both $h$ and $h^{-1}$ are strictly increasing functions on $\mathbb{R}$.
\item {The global inverse map $\omega$ is continuous at $\infty$ in the sense that
\[\lim_{|z|\to+\infty}|\omega(z)|=+\infty.\]}
\end{enumerate}
\end{prop}

The zero set of the function $f$ is characterized below. As shown in [\ref{BelBer05}], the equivalence of (1) and (5), as well as the properties (i) and (iii), can be proved alternatively using the Denjoy-Wolff analysis.
\begin{prop} \label{Hprop2}
Let $\alpha\in\mathbb{R}$. The following conditions
$(1)$-$(5)$ are equivalent: $(1)$ $\alpha\in\partial\Omega$; $(2)$ $f(\alpha)=0$; $(3)$
$g(\alpha)\in (0,b]$; $(4)$ $\Im H(\alpha+iy)>0$ for all
$y>0$; $(5)$ the angular derivative
$H^\prime(\alpha)$ exists in $[0,+\infty)$. Moreover, if
the conditions $(1)$-$(5)$ are satisfied, then we have \emph{(i)} the angular derivative of $\omega$ at $h(\alpha)$ also exists in $(0,+\infty]$ and $1/\omega^{\prime}(h(\alpha))=H^\prime(\alpha)=b-g(\alpha)$ where the case $b=g(\alpha)$ corresponds to $\omega^{\prime}(h(\alpha))=+\infty$, \emph{(ii)} the boundary value \[H(\alpha)=a+b\alpha+\int_{\mathbb{R}}\frac{1+\alpha s}{\alpha-s}\,d\rho(s),\] and \emph{(iii)} if $H^\prime(\alpha)\neq 0$ then $\partial \Omega$ is tangent to $\mathbb{R}$ at $\alpha$. If $\{\gamma(t):t\in [0,1)\}$ is a curve in $\Omega$ with $\gamma(1^{-})=\alpha$ and $H^\prime(\alpha)\neq 0$, then $H(\gamma(t)) \rightarrow_{\sphericalangle} H(\alpha)$ as $t\rightarrow 1^{-}$ if and only if $\gamma(t)\rightarrow_{\sphericalangle}\alpha$ as $t\rightarrow 1^{-}$.
\end{prop}

The next result is concerned with the analyticity on the boundary, in which the assertions (1) and (4) already appeared in  [\ref{BelBer05}].

\begin{prop} \label{Hprop3}
The following assertions hold.
\begin{enumerate} [$\qquad(1)$]
\item {For any $x\in V$, the function $H$ is conformal at the point $x+if(x)$ and $\omega$ extends analytically to a
neighborhood of the point $h(x)$.}
\item {The function $H$
extends analytically across any open interval $I\subset
\mathbb{R}\setminus\overline{V}$ by reflection and this extension is a conformal
mapping. Consequently, $\omega$ extends analytically through the
interval $H(I)\subset\mathbb{R}$ and the extension satisfies
$\omega(H(x))=x$ for $x\in I$.} \item {The function $\omega$ has a
complex analytic extension to a neighborhood of $s\in\mathbb{R}$ if and only
if both $h$ and $f$ are real analytic at the point
$h^{-1}(s)$.}
\item{Assume the angular derivative $H^{\prime}(\alpha)\neq 0$ at $\alpha\in \partial \Omega \cap \mathbb{R}$.
Then the function $H$ extends analytically to the point $\alpha$ if and only if the map $\omega$ extends analytically to the point $h(\alpha)$.}
\end{enumerate}
\end{prop}

We now apply these results to various integral forms $H$ in order to get regularity results for $\boxplus$-infinitely divisible laws and their free convolution.

\begin{exam} [Regularity of $\boxplus$-infinitely divisible laws] \label{exampleR1}
\emph{Given a nondegenerate $\mu\in\mathcal{ID}(\boxplus)$ with free L\'{e}vy-Hin\v{c}in parameters $\sigma_\mu$ and $\gamma_{\mu}$, we consider the function
\[H(z)=\gamma_{\mu}+z+\int_\mathbb{R}\frac{1+sz}{z-s}\;d\sigma_\mu(s).\] In this case we have $\omega=F_\mu$ on $\mathbb{C}^{+}$, and so $\omega$ serves as a continuous and injective extension of $F_{\mu}$ to $\mathbb{C}^{+}\cup \mathbb{R}$. The map $F_{\mu}$ is also continuous at $\infty$ in the sense of Proposition \ref{Hprop1} (6). The image \[F_\mu(\mathbb{R})=\partial \Omega=\{x+if(x):x\in\mathbb{R}\}\] is a continuous simple curve in $\mathbb{C}^{+}\cup\mathbb{R}$. Being injective, the map $F_{\mu}$ has at most one zero $s_{\mu}$ in $\mathbb{R}$. Accordingly, the set $\{s:(-\Im G_{\mu})^{*}(s)=+\infty\}$ is either the singleton set $\{s_{\mu}\}$ or the empty set. Thus, $\mu$ has a zero singularly continuous part and at most one atom, a fact that is already known in [\ref{BerVoicu93}]. The existence of $s_{\mu}$ amounts to $0\in \partial \Omega$ or, equivalently, $g(0)\leq 1$. By Proposition $\ref{Hprop2}$, this means that the angular derivative $H^{\prime}(0)$ exists and is equal to $1-g(0)$. We also have a formula for $s_{\mu}$ should it exist: \[s_{\mu}=H(F_{\mu}(s_{\mu}))=H(0)=\gamma_{\mu}-\int_{\mathbb{R}}\frac{1}{s}\,d\sigma_{\mu}(s).\] Assume that $H^{\prime}(0)$ exists in $(0,+\infty)$. Proposition \ref{Hprop2} shows that $F_{\mu}^{\prime}(s_{\mu})=1/H^{\prime}(0)$, and therefore the measure $\mu$ has an atom at $s_{\mu}$ and \[\mu(\{s_{\mu}\})=1-\int_\mathbb{R}\frac{1+s^2}{s^2}\;d\sigma_\mu(s).\]  Conversely, if $\mu(\{s_{\mu}\})>0$ then the angular derivative $H^\prime(0)$ exists and is non-zero. In summary, the following statements are equivalent:
\begin{enumerate} [$\qquad(i)$]
\item[(i)] {The zero $s_{\mu}$ exists and is an atom of $\mu$;}
\item[(ii)]
{$0\in\partial\Omega$ and the angular derivative
$H^\prime(0)>0$;}
\item[(iii)] {\[g(0)=\int_\mathbb{R}\frac{1+s^2}{s^2}\;d\sigma_\mu(s)<1.\]}
\end{enumerate} (See [Proposition 5.1, \ref{BWZ}] for a proof based on the results of [\ref{BelBer05}].) If the point $s_{\mu}$ does not exist, then $\mu$ has no atomic part. The absolutely continuous part $\mu_{\text{ac}}$ is \[d\mu_{\text{ac}}(s)=\frac{(-\Im G_{\mu})^{*}(s)}{\pi}\,d\lambda(s)=\frac{\Im F_{\mu}(s)}{\pi|F_{\mu}(s)|^2}\,d\lambda(s),\quad s \neq s_{\mu}.\] Taking the push-forward of the measures $\mu_{\text{ac}}$ and $\lambda$ by the homeomorphism $h^{-1}$, the above formula can be rewritten as \[d(\mu_{\text{ac}}\circ h)(x)=\frac{1}{\pi}\frac{f(x)}{x^2+f(x)^2}\,d(\lambda \circ h)(x), \quad x\neq 0.\] (Here we have used the inversion relationship $F_\mu(h(x))=x+if(x)$ for $x\in\mathbb{R}$.) The last formula shows first that \[\text{supp}(\mu_{\text{ac}})=h(\text{supp}(f))=h\left(\overline{V}\right)=\overline{h\left(V\right)}=\overline{\{h(x):x\in V\}}.\] In particular, by writing the open set $V$ as a countable disjoint union of open intervals, the support of $\mu$ is a countable disjoint union of closed intervals; one of which is the degenerate interval $[s_{\mu},s_{\mu}]=\{s_{\mu}\}$ if the point $s_{\mu}$ exists and is an atom lying outside of $\text{supp}(\mu_{\text{ac}})$. Secondly, denoting $D=\{h(0)\}$ if $s_{\mu}=h(0)$ exists and $D=\emptyset$ otherwise, the Radon-Nikodym derivative $d\mu_{\text{ac}}/d\lambda$ has a version $p_{\mu}$ defined by the formula:
\[p_\mu(h(x))=\frac{1}{\pi}\frac{f(x)}{x^2+f(x)^2},\quad x\in \mathbb{R}\setminus \{0\}\] and $p_\mu(h(0))=0$. The density $p_{\mu}$ is continuous on $ \mathbb{R}\setminus D$ and is positive on the open dense subset $h\left(V\right) \subset \text{supp}(\mu_{\text{ac}})$. Therefore, Proposition \ref{Hprop3} yields the analyticity of $p_{\mu}$ at any point $h(x)$ where $x \in V$. Moreover, (3.2) shows that $V=\mathbb{R}$ if and only if $g(x)>1$ for all $x\in \mathbb{R}$; in this case the density $p_\mu$ is positive and analytic everywhere on $\mathbb{R}$. Finally, since the non-existence of the point $s_{\mu}$ means
$f(0)\neq0$, we deduce from the above density formula that $p_{\mu}$
is uniformly bounded on $\mathbb{R}$ if $D=\emptyset$.}\qed
\end{exam}

We single out the conclusion about the analyticity of the density $p_{\mu}$.
\begin{prop} \label{anadensityR}
Let $\mu\in\mathcal{ID}(\boxplus)$ with L\'{e}vy measure
$\sigma_\mu\neq 0$. The measure $\mu$ is absolutely continuous with a strictly positive and analytic density everywhere on
$\mathbb{R}$ if and only if \[\int_\mathbb{R}(1+s^2)(x-s)^{-2}\;d\sigma_\mu(s)>1, \quad x\in \mathbb{R}.\]
\end{prop}

Following the ideas in Example \ref{exampleR1}, we now address the regularity questions for free convolution with a freely infinitely divisible law. Fix $\nu\in\mathcal{P}_\mathbb{R}$ and a nondegenerate measure $\mu\in\mathcal{ID}(\boxplus)$. The appropriate $H$-function to be used in this case is the map \[H(z)=z+\varphi_\mu(F_\nu(z)),\;\;\;\;\;z\in\mathbb{C}^+.\] By the free L\'{e}vy-Hin\v{c}in formula \eqref{eq:2.3}, the composition $-\varphi_\mu \circ F_\nu:\mathbb{C}^+\rightarrow \mathbb{C}^{+}\cup\mathbb{R}$ is an analytic function satisfying $\lim_{y\rightarrow\infty}\varphi_\mu(F_\nu(iy))/iy=0$. Thus, by writing the function $-\varphi_\mu \circ F_\nu$ in its Nevanlinna form, the map $H$ admits the integral representation \eqref{eq:3.1}, where $a=-\Re[(\varphi_\mu\circ F_\nu)(i)]$ and $b=1$. Note that the representing measure $\rho$ in this case is not the zero measure, because $\mu$ is nondegenerate.

We can now apply Proposition \ref{Hprop1} to the function $H$ considered above. Our first result characterizes the zero set $\partial\Omega\cap\mathbb{R}$ in terms of representing measures and vertical limits. Any inequality below means that the involved singular integrals converge and their values satisfy the estimate.
\begin{thm} \label{boundaryR}
Let $A$ be the set of all $\alpha\in\mathbb{R}$ satisfying
$F_{\nu}^{*}(\alpha)=0$ and
\[
\left[1+\int_{\mathbb{R}}\frac{1+s^{2}}{(\alpha-s)^{2}}\,d\rho_{\nu}(s)\right]\left[\int_{\mathbb{R}}\frac{1+s^{2}}{s^{2}}\,d\sigma_{\mu}(s)\right]\leq1.
\]
Let $B$ be the set of all $\alpha\in\mathbb{R}$ satisfying $G_{\nu}^{*}(\alpha)\in\mathbb{R}\setminus\{0\}$
and
\[
\left[\int_{\mathbb{R}}\frac{d\nu(s)}{(\alpha-s)^{2}}\right]\left[\int_{\mathbb{R}}\frac{1+s^{2}}{(1-sG_{\nu}^{*}(\alpha))^{2}}\,d\sigma_{\mu}(s)\right]\leq1.
\]
Let $C$ be the set of all $\alpha\in\mathbb{R}$ satisfying $G_{\nu}^{*}(\alpha)=0$
and
\[
\left[\int_{\mathbb{R}}\frac{d\nu(s)}{(\alpha-s)^{2}}\right]\left[\int_{\mathbb{R}}1+s^{2}\,d\sigma_{\mu}(s)\right]\leq1.
\]
Then the sets $A$, $B$, and $C$ are mutually disjoint and the zero
set
\[
\partial\Omega\cap\mathbb{R}=A\cup B\cup C.
\]
Furthermore, for any $\alpha\in\partial\Omega\cap\mathbb{R}$,
the angular derivative $\omega^{\prime}(h(\alpha))$ exists in $(0,+\infty)$
if and only if a strict inequality occurs in any of the above cases; otherwise, $\omega^{\prime}(h(\alpha))=+\infty$.
\end{thm}
Note that if the limit $F_\nu^{*}(\alpha)$ exists and is not zero, the monotone convergence theorem yields
\begin{multline*}
\int_{\mathbb{R}}\frac{d\nu(s)}{(\alpha-s)^{2}}=-\lim_{\epsilon\rightarrow0^{+}}\frac{\Im G_{\nu}(\alpha+i\epsilon)}{\epsilon}\\
=\lim_{\epsilon\rightarrow0^{+}}\frac{1}{\left|F_{\nu}(\alpha+i\epsilon)\right|^{2}}\frac{\Im F_{\nu}(\alpha+i\epsilon)}{\epsilon}=\frac{1}{\left|F_{\nu}^{*}(\alpha)\right|^{2}}\left[1+\int_{\mathbb{R}}\frac{1+s^{2}}{(\alpha-s)^{2}}\, d\rho_{\nu}(s)\right].
\end{multline*}
Combining this with Corollary \ref{measureatom}, we obtain an alternative description for the sets $A$, $B$, and $C$ as follows.
\begin{remark} \label{remarkR2}
\emph{The set $A$ consists of $\alpha\in\mathbb{R}$ satisfying
\[
\nu(\{\alpha\})+\mu(\{s_{\mu}\})\geq 1.
\]
The set $B$ consists of $\alpha\in\mathbb{R}$ satisfying $F_{\nu}^{*}(\alpha)\in\mathbb{R}\setminus\{0\}$
and
\[
\left[1+\int_{\mathbb{R}}\frac{1+s^{2}}{(\alpha-s)^{2}}\,d\rho_{\nu}(s)\right]\left[\int_{\mathbb{R}}\frac{1+s^{2}}{(F_{\nu}^{*}(\alpha)-s)^{2}}\,d\sigma_{\mu}(s)\right]\leq1.
\]
The set $C$ consists of $\alpha\in\mathbb{R}$ satisfying
\[
(1+\alpha^{2})\rho_{\nu}(\{\alpha\})\geq var(\mu).
\] Again, having a strict inequality in any of these cases means the finiteness of $\omega^{\prime}(h(\alpha))$. Note that the set $A$ is finite and the set $C$ is at most countable. Also, if $C$ is nonempty then points in $C$ are \emph{isolated} in the sense that to each $\alpha\in C$ there is an open disk $D$ such that $D\cap C=\{\alpha\}$. In the sequel, the points in $A$, $B$, and $C$ will be
called the \emph{boundary points of type} $A$, $B$, and $C$, respectively.}\qed
\end{remark}

We now discuss the regularity of $\mu\boxplus\nu$. In general, a free convolution of two nondegenerate measures does not have a singularly continuous part [\ref{Belinschi08}]. For $\mu\in\mathcal{ID}(\boxplus)$, this fact can be derived easily from the machinery we built so far.

\begin{exam} \label{exam3.8}
\emph{Recall that the set
\[
S=\left\{ s\in\mathbb{R}:(-\Im G_{\mu\boxplus\nu})^{*}(s)=+\infty\right\}
\]
supports the singular part $(\mu\boxplus\nu)_{\text{s}}$. Let $(\mu \boxplus\nu)_{\mathrm{sc}}$ denote the singularly continuous part of $\mu\boxplus\nu$. To each $s\in S$, consider $\alpha=\omega(s)\in\mathbb{C}^{+}\cup \mathbb{R}$, where $\omega$ is
the global inverse of $H(z)=z+(\varphi_{\mu}\circ F_{\nu})(z)$. It is easy to see that the map $\omega$ satisfies the subordination identity $G_{\mu\boxplus\nu}=G_{\nu}\circ \omega$ in $\mathbb{C}^{+}$. This implies that the point $\alpha$ cannot belong to $\mathbb{C}^{+}$; for if it does we would have \begin{eqnarray*}
-\Im G_{\nu}(\alpha) & = & -\lim_{\epsilon\rightarrow0^{+}}\Im\left[G_{\nu}(\omega(s+i\epsilon))\right]\\
 & = & -\lim_{\epsilon\rightarrow0^{+}}\Im G_{\mu\boxplus\nu}(s+i\epsilon)=(-\Im G_{\mu\boxplus\nu})^{*}(s)=+\infty,
\end{eqnarray*} a contradiction.
So we must have $\alpha\in\mathbb{R}$. Since \[\left|F_{\nu}(\omega(s+i\epsilon)) \right|=\frac{1}{\left|G_{\mu\boxplus\nu}(s+i\epsilon) \right|} \leq \frac{1}{-\Im G_{\mu\boxplus\nu}(s+i\epsilon)}\rightarrow 0\quad(\epsilon\rightarrow 0^+),\] the map $F_{\nu}$ has limit zero along the curve $\{\omega(s+i\epsilon):\epsilon>0\}$ ending at $\alpha$. Lindel\"{o}f theorem then implies $F^{*}_{\nu}(\alpha)=0$ and hence $\alpha$ must be a boundary point of type $A$. Since there are only finitely many type $A$ points and $\omega$ is injective, the set
$S$ must be a finite set. Because a singularly continuous measure does not charge
any finite set, we conclude that $(\mu\boxplus\nu)_{\text{sc}}(S)=0$, proving that $(\mu\boxplus\nu)_{\text{sc}}$ is in fact the zero measure.} \qed
\end{exam}

By Example \ref{exam3.8}, atoms of $\mu\boxplus\nu$ can only come from the $h$-image of type A boundary points (and hence there are only finitely many of them). According to [\ref{BerVoicu98}], an atom $a$ of $\mu\boxplus\nu$ is characterized by $a=b+c$ where $\mu(\{b\})+\nu(\{c\})>1$. On the other hand, it was shown in [\ref{BBG09}] that $F_{\mu\boxplus\nu}$ extends continuously to a function from $\mathbb{C}^{+}\cup\mathbb{R}$ to $\mathbb{C}^{+}\cup\mathbb{R}\cup\{\infty\}$, implying that the absolutely continuous part $(\mu\boxplus\nu)_{\mathrm{ac}}$ has a density that is continuous at points $s\in\mathbb{R}$ where $F_{\mu\boxplus\nu}(s)\neq 0$. Our next result focuses mainly on the quantitative nature of $(\mu\boxplus\nu)_{\mathrm{ac}}$.

\begin{thm} \label{thmR1}
Let $\mu\in\mathcal{ID}(\boxplus)$ and $\nu\in\mathcal{P}_\mathbb{R}$ be two nondegenerate measures.
\begin{enumerate} [$\qquad(1)$]
\item {We have the topological support $\mathrm{supp}((\mu\boxplus\nu)_{\mathrm{ac}})=h\left(\overline{V}\right)=\overline{\{h(x):x\in V\}}$. Moreover, the size of this support can be estimated from below in the sense that one always has $[\mathrm{supp}(\nu_{\mathrm{ac}})\cup \mathrm{supp}(\nu_{\mathrm{sc}})] \subset \overline{V}$.}
\item {The Radon-Nikodym derivative $d(\mu\boxplus\nu)_{\mathrm{ac}}/d\lambda$ has a version $p_{\mu\boxplus\nu}$ defined by \[p_{\mu\boxplus\nu}(h(x))=\frac{f(x)}{\pi}\int_\mathbb{R}\frac{d\nu(s)}{(x-s)^2
+f(x)^2},\quad x\in V,\] and $p_{\mu\boxplus\nu}(h(x))=0$ for $x\in \mathbb{R}\setminus V$. Defining $D=h\left(A\right)$ if $A\neq\emptyset$ and $D=\emptyset$ otherwise, the density $p_{\mu\boxplus\nu}$ is continuous everywhere on $\mathbb{R}$ except possibly on the finite set $D$, and it is analytic at the point $s=h(x)$ where $x\in V \cup (\mathbb{R}\setminus \overline{V})$.}
\item{The topological support of $\mu\boxplus\nu$ is a countable union of disjoint closed intervals, and only finitely many of these intervals can be degenerate. If $[s,s]=\{s\}$ is such an interval then $s=h(\alpha)$ for an unique $\alpha \in A$.}
\item {The formula
\[(\mu\boxplus\nu)(\{h(\alpha)\})=\frac{\nu(\{\alpha\})}
{\omega^{\prime}(h(\alpha))}\]
holds for all type $A$ boundary points $\alpha$ \emph{(}if there are any\emph{)},
where the angular derivative $\omega^\prime(h(\alpha))\in (0,+\infty]$ and the ratio on the right side is interpreted as $0$ in the case of
$\omega^\prime(h(\alpha))=+\infty$. In addition, if the set $A$ is not empty, then the unique zero $s_{\mu}$ of $F_{\mu}$ exists and $h(\alpha)=\alpha +s_{\mu}$ for all $\alpha \in A$.}
\end{enumerate}
\end{thm}

\textbf{Remarks.} Whenever the density of $\mu\boxplus \nu$ is mentioned in this paper, we shall always mean the almost everywhere continuous and locally analytic version $p_{\mu\boxplus\nu}$ from Theorem \ref{thmR1} (2). According to the decomposability results in [\ref{BW}], the density $p_{\mu\boxplus\nu}$ cannot be constantly zero between two consecutive atoms in $D$. As for the boundedness of $p_{\mu\boxplus\nu}$, assuming $\nu$ is
nondegerate, observe that $A=\emptyset$ if and only if $\mu(\left\{ a\right\} )+\nu(\left\{ b\right\} )<1$
for all $a,b\in\mathbb{R}$. By the results in [\ref{Linfty}], the density $p_{\mu\boxplus\nu}$ is uniformly bounded on $\mathbb{R}$ if $F_{\nu}$ is continuous at $\infty$ in the sense of Proposition
\ref{Hprop1} (6). Note that $F_{\mu}$ is already continuous at $\infty$,
as seen in Example \ref{exampleR1}. The $F$-transform of a compactly supported measure also has this property. Moreover, Proposition \ref{Hprop1} (6) implies that
$F_{\mu\boxplus\nu}$ is continuous at $\infty$ for any compactly supported $\nu$. If $\mu$ has property (H) (see Section 3.2), then $F_{\mu\boxplus\nu}$ is continuous at $\infty$ and $\mu\boxplus\nu$ has unbounded support.\qed

We next investigate the connectedness of $\text{supp}(\mu\boxplus\nu)$. We write $n(S)<\infty$ to indicate that a set $S\subset \mathbb{R}$ has a finite number of connected components, and when this is the case, we use $n(S)$ again to denote the number of components in $S$. Note that $n(S)<\infty$ if and only if $n(\mathbb{R}\setminus S)<\infty$.
\begin{thm} \label{nocomponentR2}
Let $\nu\in\mathcal{P}_\mathbb{R}$ and $\mu\in\mathcal{ID}(\boxplus)$ be two nondegenerate measures. Assume that $n(\emph{supp}(\sigma_\mu))<\infty$ and $n(\emph{supp}(\nu))<\infty$. Then one has $n(\emph{supp}(\mu\boxplus\nu))<\infty$, with $n\left(\emph{supp}(\mu\boxplus\nu)\right)  \leq  n\left(\emph{supp}((\mu\boxplus\nu)_{\emph{ac}})\right)+ \emph{Cardinality}(A)$ and \[n\left(\emph{supp}((\mu\boxplus\nu)_{\emph{ac}})\right) \leq 2+n(\emph{supp}(\nu))+[1+3\,n(\emph{supp}(\sigma_{\mu}))]\,n(\mathbb{R}\setminus\emph{supp}(\nu)).\] \end{thm}
Here the cardinality of the set $A$ may be estimated by $\text{Cardinality}(A) \leq [1/c]$, where the notation $[x]$ is the largest integer not exceeding $x$ and $c=1-\mu(\{s_\mu\})$.

Since $\sigma_{\mu_t}=t\,\sigma_{\mu_{1}}$ for a $\boxplus$-semigroup $(\mu_t)_{t\geq 0}$, the next result follows from Theorem \ref{nocomponentR2} immediately.
\begin{cor}
Assume that $n(\emph{supp}(\sigma_{\mu_{1}}))<\infty$. The evolution $(\mu_t\boxplus\nu)_{t\geq 0}$ starting at $\nu\in\mathcal{P}_\mathbb{R}$ satisfies $n(\emph{supp}(\mu_t\boxplus\nu))<\infty$ for all $t\geq 0$ if and only if $n(\emph{supp}(\nu))<\infty$.
\end{cor}

\textbf{Remarks.} The assumption $n(\text{supp}(\sigma_\mu))<\infty$ does imply $n(\text{supp}(\mu))<\infty$. This is obvious if $\mu$ is degenerate. We consider a nondegenerate $\mu$ and the set $V$ from Example \ref{exampleR1}. To prove $n(\text{supp}(\mu))<\infty$, it suffices to show that
$n(\mathbb{R}\setminus \overline{V})<\infty$. To this end, let $I$ be a component in the open set $\mathbb{R}\setminus\overline{V}$. Then
$g(x)=\int (1+s^2)(x-s)^{-2}\,d\sigma_\mu(s)\leq1$ for any
$x\in I$ by Proposition \ref{Hprop2}. Lemma
\ref{strictdecrR} in Section 4 shows further that $\sigma_\mu(I)=0$, and hence one has $I\subset
J$ for some component $J$ in the open set
$\mathbb{R}\setminus\mathrm{supp}(\sigma_\mu)$. Note that $I$ is the only component of
$\mathbb{R}\setminus\overline{V}$ that is contained in $J$. For if this is not the case then $g(x)>1$ at some point $x\in J$, and hence $g$ will have a local maximum in $J$. This, however, is not possible because $g$ is $C^2$ and strictly convex throughout $J$. Thus, the totality of these components $I$ is finite and \[n(\mathbb{R}\setminus\mathrm{supp}(\mu_{\text{ac}}))=n(\mathbb{R}\setminus\overline{V})\leq n(\mathbb{R}\setminus\mathrm{supp}(\sigma_\mu)).\]\qed

\subsection{Property (H)}
Recall that a measure $\mu\in \mathcal{P}_{\mathbb{R}}$ has
property (H) if for all $\nu\in \mathcal{P}_{\mathbb{R}}$, the measure $\mu\boxplus\nu$ is absolutely continuous with a positive and analytic density everywhere on $\mathbb{R}$. We first present an example to illustrate the main idea of our approach.
\begin{exam}[Boolean stable laws]  \label{exampleR}
\emph{Let $\mu\in\mathcal{ID}(\boxplus)$ be a Boolean stable law in the sense that for any constant $c>0$, there exist constants $c_1>0$ and $c_2\in\mathbb{R}$ such that \[F_{\mu}(z)+cF_{\mu}(z/c)=z+c_1F_{\mu}((z-c_2)/c_1),\quad z\in\mathbb{C}^{+}.\] It
was shown in [\ref{AriHasebe14}] that the $F$-transform of $\mu$ has the form
\[F_\mu(z)=z+e^{i\pi ab}z^{1-a},\quad z\in\mathbb{C}^{+},\] where (i)
$0<a\leq 1/2$ and $0\leq b \leq 1$, (ii)
$1/2 <a\leq 2/3$ and
$2a-1\leq ab \leq1-a$, or (iii) $a=1$ and
$b=1/2$. (The principal branch of the power function is used here.) The measure $\mu$
in the last case the Cauchy distribution, and it is known that $\mu\boxplus \nu=\mu * \nu$ for all $\nu\in \mathcal{P}_{\mathbb{R}}$, which implies $p_{\mu\boxplus\nu}>0$ on $\mathbb{R}$. In case
(i) with $b\neq0,1$, the continuity of $F_{\mu}$ on $\mathbb{C}^{+}\cup\mathbb{R}$ yields $F_\mu(0)=\lim_{x\rightarrow 0^+}(x+e^{i\pi ab}x^{1-a})=0$, $F_\mu(\mathbb{R}\setminus\{0\})\subset\mathbb{C}^+$, and the angular derivative $F_\mu^\prime(0)=+\infty$ by the explicit formula of $F_\mu$. It follows from Example \ref{exampleR1} that $\mu$ is atomless and \[\int_\mathbb{R}(1+s^2)(x-s)^{-2}\,d\sigma_\mu(s)>1, \quad x\in \mathbb{R}\setminus\{0\}.\] The last inequality says that the set $B=\emptyset$.
Since $\mu(\{0\})=0$, we have $A=\emptyset$. Note that
$y\Im[F_\mu(iy)-iy]=y\Im [e^{i\pi ab}(iy)^{1-a}]\to+\infty$ as $y\to+\infty$. After writing $F_\mu$ in its Nevanlinna form with the representing measure $\rho_{\mu}$, an application of the monotone convergence theorem to the integral in the limit shows that $m_{2}(\rho_{\mu})=+\infty$, which means that $m_{2}(\mu)=+\infty$. Thus, the set $C$ is also empty. Theorem \ref{boundaryR} then implies
the zero set $\partial \Omega \cap \mathbb{R}=A\cup B \cup C=\emptyset$ (i.e., $V=\mathbb{R}$) for any nondegenerate $\nu$, showing that $p_{\mu\boxplus\nu}>0$ on $\mathbb{R}$. A similar
argument shows that the same conclusion holds for all $a$ and $b$ in case (ii). However, the density of
$\mu$ considered above is positive and analytic
everywhere except at the origin.}\qed
\end{exam}

The preceding example shows that the conditions for property (H) are hidden in the inequalities of Theorem \ref{boundaryR} and Remark \ref{remarkR2}. In fact, we have:
\begin{thm} \label{propertyH1}
Let $\mu\in\mathcal{ID}(\boxplus)$ with L\'{e}vy measure
$\sigma_\mu \neq 0$. Then the density $p_{\mu\boxplus\nu}$ is positive and analytic everywhere on $\mathbb{R}$ for every $\nu\in\mathcal{P}_{\mathbb{R}}$ if and
only if
\[
-\lim_{\epsilon\rightarrow0^+}\frac{\Im\varphi_\mu(x+i\epsilon)}{\epsilon}=\int_\mathbb{R}\frac{1+s^2}{(x-s)^2}\,d\sigma_{\mu}(s)>1,\quad x\in\mathbb{R},
\] and
\[
-\lim_{y\rightarrow\infty}y\Im\varphi_\mu(iy)=\int_\mathbb{R}1+s^2\,d\sigma_\mu(s)=+\infty.
\]
\end{thm}

Proposition \ref{prop2.4} implies a qualitative version of the preceding result.

\begin{cor} \label{propertyH2}
Given a nondegenerate measure $\mu\in\mathcal{ID}(\boxplus)$, the following statements are equivalent:
\begin{enumerate} [$\qquad(1)$]
\item {The measure $\mu$ has property $(\mathrm{H})$.}
\item {The measure $\mu$ has a positive density everywhere on
$\mathbb{R}$ and $m_2(\mu)=+\infty$.}
\item {The function $G_\mu$
extends analytically to $\mathbb{R}$, $\Im G_\mu(x)<0$ for all
$x\in\mathbb{R}$, and $m_2(\mu)=+\infty$.}
\item {The function $F_\mu$
extends analytically to $\mathbb{R}$, $\Im F_\mu(x)>0$ for all
$x\in\mathbb{R}$, and $m_2(\rho_{\mu})=+\infty$.}
\end{enumerate}Moreover, if $\mu$ has property $(\mathrm{H})$, then
$\mu\boxplus\nu$ also has property \emph{(H)} and $m_2(\mu\boxplus\nu)=+\infty$ for every $\nu\in\mathcal{P}_\mathbb{R}$.
\end{cor}

The next result follows easily from the fact $\sigma_{\mu_t}=t\,\sigma_\mu$.
\begin{cor} \label{propertyH3}
A $\boxplus$-semigroup $\left(\mu_t\right)_{t\geq 0}$ with L\'{e}vy measure $\sigma_{\mu_1}$ has property $(\mathrm{H})$ if
and only if $m_2(\sigma_{\mu_1})=+\infty$ and \[\int_\mathbb{R}\frac{1+s^2}{(x-s)^2}\,d\sigma_{\mu_1}(s)=+\infty,\quad x\in\mathbb{R}.\]
\end{cor}

\begin{exam} [Property (H) in domains of attraction]
\emph{Interestingly, all $\boxplus$-stable processes except the asymmetric ones and the free Brownian motion have property (H). This is because the Voiculescu transform $\varphi_{\nu_{a,\theta}}$ of a freely stable law $\nu_{a,\theta}$ is parameterized by its stability index $a\in (0,2]$ and the asymmetric coefficient $\theta \in [-1,1]$, and such an $\varphi_{\nu_{a,\theta}}$ must belong to the following list up to dilation and translation [\ref{BerPata99}]:
\begin{enumerate} [$\qquad(1)$]
\item[(i)] {$\varphi_{\nu_{a,\theta}}(z)=-\left[i+\theta \tan(a \pi/2)\right]i^{a-1}z^{1-a}$ when
$a\in (0,1)\cup (1,2)$;}
\item[(ii)]
{$\varphi_{\nu_{a,\theta}}(z)=2 \theta \log z -i \pi (1+\theta)$ when $a=1$;}
\item[(iii)] {$\varphi_{\nu_{a,\theta}}(z)=1/ z$ when $a=2$.}
\end{enumerate} Therefore, according to Theorem \ref{propertyH1} and Corollary \ref{propertyH3}, if $\theta \neq \pm 1$ and $a \neq 2$, then $\nu_{a,\theta}$ and the $\boxplus$-semigroup generated by $\nu_{a,\theta}$ both have property (H) from these explicit formulae. The next two examples show that the property (H) is not a topological property under weak convergence.
\begin{enumerate} [$\qquad(1)$]
\item
{Let $\sigma_{\mu}$ be an absolutely continuous measure defined by the density $f(s)=1$ for $s \in [-1,1]$ and $f(s)=|s|^{-3}$ for $|s|>1$ and let $\gamma_{\mu}=0$. The tail-sums of $\sigma_{\mu}$ satisfy $\sigma_{\mu}\left(\{s\in \mathbb{R}:|s|>y\}\right)=y^{-2}$ for $y\geq1$, so that $m_2(\sigma_{\mu})=+\infty$ and the measures \[\mu_n=D_{1/\sqrt{n \log n}}\,\mu^{\boxplus n}\] converge weakly to the standard semicircular law $\nu_{2,0}$ by the free central limit theorem [\ref{Pata}]. While the law $\nu_{2,0}$ does not have property (H), the attracted measure $\mu$ does, because $m_2(\mu)=+\infty$ and
\[
\int_{\mathbb{R}}\frac{1+s^{2}}{(x-s)^{2}}\,d\sigma_{\mu}(s)\geq f(|[x]|+1)\int_{[[x],[x]+1)}\frac{d\lambda(s)}{(x-s)^{2}}=+\infty
\]
for all $x\in\mathbb{R}$. In fact, every measure $\mu_n$ has property (H).
}
\item
{If $\nu_{a,\theta}$ is a stable law with property
(H), then its domain
of attraction always contains an $\boxplus$-infinitely divisible law which does not have property (H). To show this, fix $a\in(0,1)\cup(1,2)$ and $\theta\in(-1,1)$, and let $m\in\mathbb{N}$ be large enough so that $am>3$ and $(2-a)m>1$. Choose the numbers $p=(1-\theta)/2$, $q=(1+\theta)/2$, $c=a\pi/(2|\cos(a\pi/2)|)$,
$\gamma_{\mu}=0$. Define the L\'{e}vy measure $\sigma_{\mu}=\sigma_{\mu}^{+}+\sigma_{\mu}^{-}$
by specifying
\[
d\sigma_{\mu}^{+}(s)=pc\,s^{-(a+1)}\,d\lambda(s),\quad1\leq s<+\infty,
\]
and
\[
d\sigma_{\mu}^{-}(s)=(q/p)\sum\nolimits_{k=1}^{\infty}\sigma_{\mu}^{+}\big([k^m,(k+1)^m)\big)\,\delta_{-k^m}.
\]
Then the tail-sum $\sigma_{\mu}(\left\{ s\in\mathbb{R}:|s|>y\right\} )$
is a regularly varying function of index $-a$. By the results
in [\ref{BerPata99}], there are constants $a_{n}\in\mathbb{R}$ and $b_{n}>0$
such that $b_{n}\rightarrow0$ and the measures
\[
\mu_{n}=D_{b_{n}}\mu^{\boxplus n}\boxplus\delta_{a_{n}}
\]
converge weakly to the stable law $\nu_{a,\theta}$ as $n\rightarrow\infty$.
We also have $b_{n}=n^{-1/a}\ell(n)$ where $\ell(n)$ is a slowly
varying sequence in $\mathbb{N}$. Note that the L\'{e}vy measure of $\mu_{n}\in\mathcal{ID}(\boxplus)$
is given by
\[
d\sigma_{\mu_{n}}(s)=n\,\frac{b_{n}^{2}+s^{2}}{1+s^{2}}\,d\sigma_{\mu}(s/b_{n}),
\] so we have
\begin{align*}
\int_{-\infty}^{0}\frac{1+s^{2}}{(x-s)^{2}}\,d\sigma_{\mu_{n}}(s)&=nb_n^2\int_{-\infty}^0\frac{1+s^2}{(x-b_ns)^2}\;
d\sigma_\mu^-(s) \\
&\leq
\int_{-\infty}^{-1}\frac{2ns^{2}}{(x/b_n-s)^{2}}\,d\sigma_{\mu}^-(s).
\end{align*} Consider $x_n=-b_n[(n+1)^m+n^m]/2$ and $\widetilde{x}_n=x_n/b_n$ for $n\in\mathbb{N}$. Since
\begin{align*}
\sigma_{\mu}^{+}([k^m,(k+1)^m))&=pcm\int_k^{k+1}s^{-(am+1)}\;d\lambda(s) \\
&\leq pcmk^{-(am+1)},\;\;\;\;\;k\in\mathbb{N},
\end{align*} and
\[|\widetilde{x}_n+k^m|\geq\frac{(n+1)^m-n^m}{2}\geq\frac{m\cdot n^{m-1}}{2}, \quad 1\leq k\leq n,\] it follows that
\begin{align*}
\int_{[-n^m,-1]}\frac{ns^{2}\,d\sigma_{\mu}^{-}(s)}{(\widetilde{x}_n-s)^{2}}&
=\frac{nq}{p}\sum_{k=1}^n\frac{k^{2m}\sigma_{\mu}^{+}([k^m,(k+1)^m))}
{(\widetilde{x}_n+k^m)^2} \\
&\leq\frac{4qc}{m\cdot n^{2m-3}}\sum_{k=1}^nk^{(2-a)m-1} \\
&\leq\frac{4qc}{m\cdot n^{2m-3}}\int_1^{n+1}s^{(2-a)m-1}\;d\lambda(s) \\
&\leq\frac{2^{(2-a)m+2}qc}{(2-a)m^2}\cdot\frac{1}{n^{am-3}}.
\end{align*} The inequalities
\[0<\frac{k^m}{\widetilde{x}_n+k^m}\leq\frac{(n+1)^m}{\widetilde{x}_n+(n+1)^m}\leq2(n+1),\;\;\;\;\;k\geq n+1,\] then imply that
\begin{align*}
\int_{(-\infty,-n^m)}\frac{ns^{2}\,d\sigma_{\mu}^{-}(s)}{(\widetilde{x}_n-s)^{2}}
&=n\sum_{k=n+1}^\infty\left(\frac{k^m}{\widetilde{x}_n+k^m}\right)^2\sigma_\mu^-(\{-k^m\}) \\
&\leq16n^3\sigma_\mu^-((-\infty,-n^m))=\frac{16\widetilde{\ell}(n^m)}{n^{am-3}}
\end{align*} for some slowly varying function $\widetilde{\ell}$. These findings yield that
\[\lim_{n\to\infty}\int_{-\infty}^{0}\frac{1+s^{2}}{(x_{n}-s)^{2}}\,d\sigma_{\mu_{n}}(s)=0.\] On the other hand, the fact that
\[|x_n|\geq n^mb_n=n^{m-1/a}\ell(n)\geq n\] for all sufficiently large $n$ and Proposition 5.8 in [\ref{BerPata99}] imply
that
\begin{eqnarray*}
\int_{0}^{\infty}\frac{1+s^{2}}{(x_{n}-s)^{2}}\,d\sigma_{\mu_{n}}(s) & = & nb_{n}^{2}\int_{1}^{\infty}\frac{1+s^{2}}{(x_{n}-b_{n}s)^{2}}\,d\sigma_{\mu}^{+}(s)\\
 & \leq & 2nb_{n}^{2}\int_{1}^{\infty}\frac{s^{2}}{n^{2}+(b_{n}s)^{2}}\,d\sigma_{\mu}^{+}(s)\\
 & \sim & \frac{a\pi}{\sin(a\pi/2)}n\sigma_{\mu}^{+}((nb_{n}^{-1},+\infty))\\
 & = & \frac{\pi pc}{\sin(a\pi/2)}\left[\frac{\ell(n)}{n}\right]^{a}\rightarrow0
\end{eqnarray*}
as $n\rightarrow\infty$. We conclude at once that when $n$ is sufficiently
large,
\[
\int_{\mathbb{R}}(1+s^{2})(x_{n}-s)^{-2}\,d\sigma_{\mu_{n}}(s)<1
\]
and the measure $\mu_{n}$ does not have property (H)
even though $m_{2}(\mu_{n})=+\infty$. Finally, we remark that the case $a=1$ corresponds to Cauchy distributions, and it is easy to construct examples of this sort.\qed}
\end{enumerate} }
\end{exam}

\begin{exam} [Normal and Askey-Wimp distributions] \emph{The Askey-Wimp laws $\mu_{c}$
is a family of $\boxplus$-infinitely divisible laws with positive
density everywhere on $\mathbb{R}$ and $m_{2}(\mu_{c})<+\infty$
(see [\ref{normal}] for a proof of the free infinite divisibility). They
are parameterized by $c\in(-1,0]$, with $\mu_{0}$
being the standard normal distribution. For such a measure $\mu_{c}$
and a nondegenerate $\nu\in\mathcal{P}_{\mathbb{R}}$, type $A$ and $B$
boundary points do not exist because $p_{\mu_c}>0$ on $\mathbb{R}$. So, the measure $\mu_{c}\boxplus\nu$
is always absolutely continuous with $\text{supp}(\mu_{c}\boxplus\nu)=\mathbb{R}$ and a continuous density everywhere on $\mathbb{R}$.
The density $p_{\mu_{c}\boxplus\nu}$ would vanish at the point $s_{0}=h(\alpha)$
as long as a type $C$ point $\alpha$ shows up. On the other hand, the density $p_{\mu_{c}\boxplus\nu}$
could be positive and analytic everywhere on $\mathbb{R}$ even for
some extremely singular $\nu$ such as
\[
\nu=\frac{18}{\pi^{4}}\sum\nolimits _{m\in\mathbb{Z}}\frac{1}{m^2}\left\{ \sum\nolimits _{n\in\mathbb{N}}\frac{1}{2^{n}n^{2}}\sum\nolimits_{j=1}^{2^{n}}\delta_{m+j2^{-n}}\right\}
\]
or
\[
d\nu(s)=\frac{18}{\pi^{4}}\sum\nolimits _{m\in\mathbb{Z}}\frac{1}{m^{2}}I_{[m,m+1]}(s)\left\{ \sum\nolimits _{n\in\mathbb{N}}\frac{1}{2^{n}n^{2}}\sum\nolimits _{j=1}^{2^{n}}d\tau\left(s-m-j2^{-n}\right)\right\},
\]
where $I_{[m,m+1]}$ is the indicator function of the interval $[m,m+1]$
and $\tau$ denotes the Cantor distribution on $[0,1]$. This is because
these examples satisfy
\[
\int_{\mathbb{R}}(x-s)^{-2}\,d\nu(s)=+\infty,\quad x\in \mathbb{R},
\]
(see [\ref{BSimon}] for the divergence of this singular integral), which
rules out the existence of type $C$ boundary points. In contrast,
for the Bernoulli trial $\nu=(\delta_{-1}+\delta_{1})/2$ or the absolutely
continuous measure
\[
d\nu(s)=(3/8)[s^{2}I_{[-1,1]}(s)+s^{-2}I_{\mathbb{R}\setminus[-1,1]}(s)]\,d\lambda(s),
\]
the point $\alpha=0$ is the only type $C$ boundary point
for $\mu_{0}\boxplus\nu$, with the angular derivative $\omega^{\prime}(h(0))=+\infty$.
Then $p_{\mu_{0}\boxplus\nu}(h(0))=0$, and $p_{\mu_{0}\boxplus\nu}$ is not analytic at $h(0)$ by Proposition \ref{nonanaR} (2) in the next section.}\qed
\end{exam}

\subsection{Analyticity at zeros}
We now examine the analyticity
of the density $p_{\mu\boxplus\nu}$ at a zero $s_0=h(\alpha)$ in its support, where $\alpha=\omega(s_0)\in \partial \Omega \cap \mathbb{R}$.

\begin{prop} \label{nonanaR}
The density $p_{\mu\boxplus\nu}$ is not analytic at the zero
$s_0$ in the following two situations:
\begin{enumerate} [$\qquad(1)$]
\item {There is a sequence $s_n \in \mathbb{R}\setminus \{s_0\}$ such that $p_{\mu\boxplus\nu}(s_n)=0$
for all $n$ and $\lim_{n\rightarrow \infty}s_n=s_0$.}
\item {The angular
derivative $\omega^{\prime}(s_0)=+\infty$.}
\end{enumerate}
\end{prop}

Thus, we shall assume that $s_0$ is an isolated zero of $p_{\mu\boxplus\nu}$ and $\omega^{\prime}(s_0)<+\infty$ from now on. (The latter finiteness condition means that a strict inequality occurs in Theorem \ref{boundaryR}.)

In the next result, the value of the vertical limit $F^{*}_{\nu}(\alpha)$ for $\alpha \in C$ is interpreted as the point of $\infty$, so that it makes sense to talk about the analyticity at $F^{*}_{\nu}(\alpha)$ for a measure on $\mathbb{R}$ (see Definition \ref{def2.7}).

\begin{thm} \label{analyticR1}
Let $s_0=h(\alpha)$ be an isolated zero of $p_{\mu\boxplus\nu}$ in its support such that $\alpha$ satisfies one of the strict inequalities in \emph{Theorem \ref{boundaryR}}. Assume that the L\'{e}vy measure $\sigma_{\mu}$ is analytic at the point $F^{*}_{\nu}(\alpha)$. Then, depending on the type of the boundary points, the analyticity of $p_{\mu\boxplus\nu}$ at $s_0$ is equivalent to one of the two conditions:
\begin{enumerate} [$\qquad(1)$]
\item {The measure $\nu$ is meromorphic at the point
$\alpha$ when $\alpha$ is of type $A$.}
\item {The measure $\nu$ is analytic at the point $\alpha$ when $\alpha$ is of type $B$ or $C$.}
\end{enumerate}
The zero $s_0$ is an atom of $\mu\boxplus\nu$ in the case \emph{(1)}. In the case \emph{(2)}, we have: \[\alpha=s_0-\gamma_{\mu}+\int_{\mathbb{R}}\frac{1+sF^{*}_{\nu}(\alpha)}{s-F^{*}_{\nu}(\alpha)}\,d\sigma_{\mu}(s)\] when $\alpha \in B$, and $\alpha=s_0-\gamma_{\mu}-m_1(\sigma_{\mu})$ when $\alpha \in C$. \end{thm}

\textbf{Remarks.} One has the derivative $p_{\mu\boxplus\nu}^\prime(s_0)=0$ if $p_{\mu\boxplus\nu}$ is analytic at $s_0$. This is simply because $s_0$ is a local minimum of the density. \qed

\begin{exam}[Free Poisson Process] \label{exampleR3}
\emph{Let $\mu_{t}$ be the Mar\v{c}enko-Pastur
law associated with parameter $t>0$. Thus, we have for any $t>0$,
\[
\gamma_{\mu_{t}}=t/2,\quad\sigma_{\mu_{t}}=(t/2)\,\delta_{1},
\]
and $s_{\mu_{t}}=0$ if $t\leq1$.
\begin{enumerate}
\item (\textbf{Type} $A$) Consider $t<1/2$ and the probability measure
\[d\nu(s)=2t\,\delta_{\alpha}+(1-2t)I_{[a,b]}(s)\,d\lambda(s)\]
supported on $[a,b]$, where $b=a+1\in\mathbb{R}$ and $\alpha=(a+b)/2$. Then
\[\int_{\mathbb{R}}(1+s^{2})s^{-2}\,d\sigma_{\mu_t}(s)=t<\nu(\{\alpha\}).\]
Note that for any $x\in[a,b]\setminus\left\{ \alpha\right\} $, we
have
\[\int_{\mathbb{R}}\frac{d\nu(s)}{(x-s)^{2}}\geq(1-2t)\int_{a}^{b}\frac{d\lambda(s)}{(x-s)^{2}}=+\infty,\]
implying that $x\in V$. Since $H(\alpha)=\alpha$ and $H'(\alpha)=1/2$, it follows that $h(\alpha)=\alpha$ is an
isolated zero of $p_{\mu_{t}\boxplus\nu}$ in its support and $\omega^{\prime}(\alpha)<+\infty$.
Clearly, $\sigma_{\mu_{t}}$ is analytic at $F_{\nu}^{*}(\alpha)=0$
and $\nu$ is meromorphic at the type $A$ point $\alpha$, and so the density
$p_{\mu_{t}\boxplus\nu}$ is analytic at the atom $\alpha$.
\item (\textbf{Type} $B$) Consider $t<9/4$ and let $\nu$ be the absolutely
continuous measure
\[
d\nu(s)=3^{-1}s^{2}I_{[-1,2]}(s)\,d\lambda(s).
\]
For any $x\in[-1,2]$, one has
\[
\int_{\mathbb{R}}(x-s)^{-2}\,d\nu(s)=3^{-1}\int_{-1}^{2}s^{2}(x-s)^{-2}\,d\lambda(s)=\begin{cases}
+\infty & \text{if }x\neq0;\\
1 & \text{if }x=0,
\end{cases}
\]
showing that $[-1,0)\cup(0,2]\subset V$. Proposition \ref{prop2.1} shows further
that
\[G_{\nu}^{*}(0)=-\int_{\mathbb{R}}s{}^{-1}\,d\nu(s)=-\frac{1}{2}.\]
Since $t<9/4$, the point $\alpha=0$ is of type $B$ and $s_{0}=h(0)$
is an isolated zero of $p_{\mu_{t}\boxplus\nu}$ satisfying $\omega^{\prime}(s_{0})<+\infty$
in its support. Since $\sigma_{\mu_t}$ and $\nu$ are analytic at $F_{\nu}^{*}(\alpha)=-2$
and $\alpha$ respectively, the density $p_{\mu_{t}\boxplus\nu}$
is analytic at $s_{0}=2t/3$.
\item (\textbf{Non-analytic Type} $B$) Take $t<25/126$ for the absolutely
continuous measure $\nu$ defined by the density
\[
q(s)=(12/7)\left[s^{2}I_{[-1,0]}(s)+s^{3}I_{[0,1]}(s)\right].
\]
Then $\alpha=0$ is a type $B$ boundary point and $s_{0}=7t/5$ is
an isolated zero of $p_{\mu_{t}\boxplus\nu}$ in its support such
that $\omega^{\prime}(s_{0})<+\infty$. Also, $\sigma_{\mu_t}$ is analytic
at $F_{\nu}^{*}(\alpha)=7/2$. Since the density $q$ is not $C^{3}$
at the origin, the density $p_{\mu_{t}\boxplus\nu}$ is not analytic
at $s_{0}$.
\item (\textbf{Type} $C$) Consider $t<1$ and let $\nu$ be the absolutely
continuous measure
\[
d\nu(s)=(3/8)\left[s^{2}I_{[-1,1]}(s)+s^{-2}I_{\mathbb{R}\setminus[-1,1]}\right]\,d\lambda(s).
\]
This time we have $\mathbb{R}\setminus V=\left\{ 0\right\} $ and
$\alpha=0$ is a type $C$ boundary point. So, the point $s_{0}=t$
is the unique zero of $p_{\mu_{t}\boxplus\nu}$, with $\text{supp}((\mu_{t}\boxplus\nu)_{\text{ac}})=\mathbb{R}$
and $\omega^{\prime}(s_{0})<+\infty$. The density $p_{\mu_{t}\boxplus\nu}$
is analytic at $s_{0}$ in this case. \qed
\end{enumerate}}
\end{exam}

\setcounter{equation}{0}
\section{The Proofs}
\subsection{Global inversion and applications} The following appeared in [\ref{Huang14}].
\begin{lem} \label{strictdecrR}
Let $\rho\in \mathcal{M}_{\mathbb{R}}\setminus \{0\}$. For every
$x\in\mathbb{R}$, the continuous function
\[\mathcal{I}_x(y)=\int_\mathbb{R}\frac{1+s^2}{(x-s)^2+y^2}\;d\rho(s)\] is strictly
decreasing on $(0,\infty)$,
$\lim_{y\to+\infty}\mathcal{I}_x(y)=0$, and
\[\lim_{y\rightarrow 0^{+}}\mathcal{I}_x(y)=\sup_{y > 0}\mathcal{I}_x(y)=
\int_\mathbb{R}\frac{1+s^2}{(x-s)^2}\;d\rho(s)\in(0,+\infty].\] If $\sup_{x\in I}\sup_{y > 0}\mathcal{I}_x(y)<+\infty$
for some open interval $I$, then $\rho(I)=0$.
\end{lem}

\addvspace{0.1in}

\textit{Proof} of $\mathbf{Proposition\;\;\ref{Hprop1}}$: We begin with the properties of the map $f$. The definition of $f$ implies that
\begin{equation}\label{f1}
\mathcal{I}_x(f(x))=\int_\mathbb{R}\frac{1+s^2}{(x-s)^2+f(x)^2}\;d\rho(s)\leq
b,\;\;\;\;\;x\in\mathbb{R}.
\end{equation}
Note that the equality in \eqref{f1} holds whenever $f(x)>0$. Indeed, if $f(x)>0$ and yet $\mathcal{I}_x(f(x))<b$, then there is a $\delta>0$ such that $\mathcal{I}_x(f(x))<b-\delta<b\leq \mathcal{I}_x(f(x)/2)$. The intermediate value theorem implies $\mathcal{I}_x(y)=b-\delta$ for some $0<y<f(x)$, a contradiction to the definition of $f(x)$.

The preceding observation and the strict monotonicity of $\mathcal{I}_x$ imply that if $f(x)>0$ then \[b=\mathcal{I}_x(f(x))<\int_{\mathbb{R}}\frac{1+s^2}{(x-s)^2}\,d\rho(s)=g(x).\] Conversely, $f(x)=0$ implies $\mathcal{I}_x(y)<b$ for all $y>0$, whence $g(x)=\sup_{y>0}\mathcal{I}_x(y)\leq b$. So, we have $V=\{x\in \mathbb{R}:g(x)>b\}$.

We already have $\Omega\supset \{x+iy\in \mathbb{C}^{+}:y>f(x)\}$, because the set $\Omega$ contains the vertical array $\{z+iy:y>0\}$ for any $z\in\Omega$. Conversely, if $x+iy\in \Omega\subset \mathbb{C}^{+}$ and $f(x)>0$ then $\mathcal{I}_x(y)<b=\mathcal{I}_x(f(x))$, showing that $f(x)<y$. It follows that $\Omega= \{x+iy\in \mathbb{C}^{+}:y>f(x)\}$.

Being a supremum of continuous maps, $g$ is lower semi-continuous. We now verify the continuity of $f$. Consider the function $F(x,y)=\mathcal{I}_x(y)-b$ on $\mathbb{R}\times(0,\infty)$, and let $x_0\in V$. Clearly, $F$, $\partial F/\partial x$ and $\partial F/\partial y$ are all continuous on some neighborhood of $(x_0,f(x_0))$. Since $F(x_0,f(x_0))=0$ and $\partial F/\partial y (x_0,f(x_0))\neq 0$, it follows from the implicit function theorem that $f$ is continuous at $x_0$. If $f$ is not continuous at some $x_0\in\mathbb{R}\backslash V$, i.e., there exist points $x_n$ so that $x_n\to x_0$ and $\inf f(x_n)\geq\epsilon>0$. Then dominated convergence theorem implies that
\[b=\lim_{n\to\infty}\mathcal{I}_{x_n}(f(x_n))\leq\lim_{n\to\infty}\int_\mathbb{R}\frac{(1+s^2)d\rho(s)}{(x_n-s)^2+\epsilon^2}
=\int_\mathbb{R}\frac{(1+s^2)d\rho(s)}{(x_0-s)^2+\epsilon^2}<b,\] which is clearly a contradiction. Hence $f$ is continuous on $\mathbb{R}$.

The continuity of $f$ implies $\partial\Omega=\left\{ x+if(x):x\in\mathbb{R}\right\} $.
Indeed, notice first that
\[
\mathcal{I}_{x}(f(x)+\epsilon)<b<\mathcal{I}_{x}(f(x)-\epsilon),\quad x\in V,\;\epsilon\in(0,f(x)),
\]
and $\partial\Omega\cap\mathbb{R}=f^{-1}(\left\{ 0\right\} )=\mathbb{R}\setminus V$.
This shows that $\left\{ x+if(x):x\in\mathbb{R}\right\} \subset\partial\Omega$.
Conversely, every boundary point $x+iy\in\partial\Omega$ satisfies $y\leq f(x)$. Moreover, for any $\Omega\ni z_{n}=x_{n}+iy_{n}\rightarrow x+iy$,
one has
\[
y=\lim_{n\rightarrow\infty}y_{n}\geq\lim_{n\rightarrow\infty}f(x_{n})=f(x),
\]
showing that we must have $y=f(x)$.

Another consequence of the continuity of $f$ is that the open set $\Omega$ is path-connected and hence connected. This is because for any two distinct points $z_j=x_j+iy_j$ ($j=1,2$) in $\Omega$ (say, $x_1<x_2$), there is a continuous vertical-horizontal-vertical path joining $z_1$ and $z_2$ through the coordinates $(x_1,m)$ and $(x_2,m)$ in $\Omega$, where $m=1+\max \{f(x):x_1\leq x \leq x_2\}$. A similar argument shows that any closed loop in $\Omega$ can be continuously retracted to a point in $\Omega$, proving that $\Omega$ is in fact simply connected.

We next address the boundary behaviour of $H$. For $z_{1}\neq z_{2}$ in $\Omega$, an application of
Cauchy-Schwarz inequality and (\ref{f1}) imply
\begin{equation} \label{eq:4.2}
\left|\frac{H(z_{1})-H(z_{2})}{z_{1}-z_{2}}-b\right|=\left|\int_{\mathbb{R}}\frac{1+s^{2}}{(z_{1}-s)(z_{2}-s)}\,d\rho(s)\right| \leq b.
\end{equation}
So the map $H$ is uniformly Lipschtz in $\Omega$ and as a such
it has a continuous extension to the closure $\overline{\Omega}$,
for which we still denote by $H$. The estimate \eqref{eq:4.2} extends by continuity to any pair of distinct points on $\overline{\Omega}$, and therefore
\[
\left|H(z_1)-H(z_2)\right|\leq2b|z_1-z_2|,\qquad z_1,z_2\in\overline{\Omega}.
\]

Before we go further, let us additionally assume that the measure $\rho$ in (\ref{f1}) is compactly supported on $\mathbb{R}$, say, $\mathrm{supp}(\rho)\subset[-L,L]$ for some finite $L>0$. We will show that under this extra assumption, $H$ is an injective map from $\Omega$ onto $\mathbb{C}^+$. It is equivalent to showing that for an arbitrary but fixed point $w\in\mathbb{C}^+$, the equation $H(z)=w$ has one and only one solution in $\Omega$. To show this, first observe that for all sufficiently large $|x|$ and for all $y>0$, we have
\[|bx-\Re H(x+iy)|=\left|a+\int_\mathbb{R}\frac{s(x^2+y^2-1)+(1-s^2)x}{(x-s)^2+y^2}\;d\rho(s)\right|\leq|a|+2L\rho(\mathbb{R}).\] We also have
$\Im H(x+iy)\geq by-(1+L^2)\rho(\mathbb{R})/y$ for all $x\in\mathbb{R}$ and $y>0$. These observations allow us to find a closed path $\gamma$ in $\Omega$ so that $w$ lies insides the region enclosed by the path $H(\gamma)$ and $H(\gamma)$ winds around the point $w$ only one time. For instance, $\gamma$ can be a rectangle-like path consisting of one curve lying above but close to $\partial\Omega$ as its base boundary, two vertical segments with sufficiently large $x$ and negative large $x$ coordinates as its opposite sides, and a horizontal segment with sufficiently large $y$-coordinates as its upper side.
Then
\[\frac{1}{2\pi i}\int_\gamma\frac{H'(z)}{H(z)-w}\;dz=\frac{1}{2\pi i}\int_{H(\gamma)}\frac{d\xi}{\xi-w}=1\] with Cauchy's argument principle shows that the function $H(z)-w$ has exactly one zero inside $\gamma$ as well as inside $\Omega$. We have shown that $H:\Omega\to\mathbb{C}^+$ is bijective and has an analytic inverse $\omega$ from $\mathbb{C}^+$ onto $\Omega$ provided that $\rho$ is compactly supported.

Now we return to the general case $\rho$. Choose positive measures with finite support $\rho_n$ converging to $\rho$ weakly, and let $H_n$, $\omega_n$, and $\Omega_n$ be the corresponding functions and sets associated with $\rho_n$. Then $H_n\to H$ uniformly on compact subsets of $\mathbb{C}^+$, and $\{\omega_n\}$ is normal by Montel's Theorem. We assume, without loss of generality, that $\omega_n\to\omega$ uniformly on compacta for some analytic function $\omega$ on $\mathbb{C}^+$.
For any but fixed $z\in\Omega$, $\lim H_n(z)=H(z)\in\mathbb{C}^+$ shows that $H_n(z)\in\mathbb{C}^+$ for all sufficiently large $n$, and so the uniform convergence yields that $\omega(H(z))=\lim\omega_n(H_n(z))=z$. Being an open mapping, the image $\omega(\mathbb{C}^+)$ of $\omega$ never touches the real line, i.e., $\omega(\mathbb{C}^+)\subset\mathbb{C}^+$. We conclude from the uniform convergence again that $H(\omega(\xi))=\lim H_n(\omega_n(\xi))=\xi$ for any $\xi\in\mathbb{C}^+$.
Consequently, $H$ is an injective map from $\Omega$ onto $\mathbb{C}^+$ and has $\omega$ as its inverse.

The boundary behaviour of the map $\omega$ is proved as follows. We consider the conjugation $W=M\circ\omega\circ M^{-1}$
under the Cayley transform $M(z)=(z-i)/(z+i)$, so that the function
$W$ maps the open disk $\mathbb{D}$ conformally onto $M(\Omega)$.
Note that
\[
\left|1-M(x+if(x))\right|= \frac{2}{\sqrt{x^{2}+[f(x)+1]^{2}}}\leq\frac{2}{|x|}\rightarrow0\quad(|x|\rightarrow\infty).
\]
Therefore, the image of the continuous graph $\{x+if(x):x\in\mathbb{R}\} $
under the transform $M$ is a non-self-intersecting continuous curve
in the closed disk $\overline{\mathbb{D}}$ such that both ends of this curve meet at the point $1$. In other words, the image $W(\mathbb{D})=M(\Omega)$
is a bounded domain whose boundary is a Jordan curve. By Carath\'{e}odory's extension theorem [\ref{Rudin}], the conformal map $W$ extends continuously to a homeomorphism
from the unit circle $\partial\mathbb{D}$ onto the Jordan curve $\partial M(\Omega)$.
This implies that $\omega$ extends continuously to a homeomorphism
from $\mathbb{R}$ to $\partial\Omega$. Of course, the extension, still denoted by $\omega$, satisfies \[2b|\omega(z_{1})-\omega(z_{2})|\geq|z_{1}-z_{2}|, \quad z_{1},z_{2}\in\mathbb{C}^{+}\cup\mathbb{R}.\] The inversion relations between the maps $H$ and $\omega$ are extended to the topological boundaries $\partial \Omega$ and $\mathbb{R}$ by continuity.

The map $h(x)=H(x+if(x))$, $x\in\mathbb{R}$, is simply
the inverse of the homeomorphism $\mathbb{R}\ni s\mapsto\Re\omega(s)$.

Now, we verify that both $h(x)=H(x+if(x))$ and $h^{-1}(s)=\Re\omega(s)$
are strictly increasing on $\mathbb{R}$. It suffices to prove this
only for $h^{-1}$. To this end, observe from (\ref{eq:4.2}) that the difference quotient $H(z_{1})-H(z_{2})/(z_{1}-z_{2})$
belongs to the closed disk $\left\{ z\in\mathbb{C}:|z-b|\leq b\right\} $
for any $z_{1}\neq z_{2}$ in $\overline{\Omega}$. In particular,
if $z_{1}=\omega(s_{1})$ and $z_{2}=\omega(s_{2})$ for $s_{1}>s_{2}$,
we have
\begin{equation}\label{eq:5.3}
0<\Re\left[\frac{H(z_{1})-H(z_{2})}{z_{1}-z_{2}}\right]=\frac{s_{1}-s_{2}}{|\omega(s_{1})-\omega(s_{2})|^{2}}\left[h^{-1}(s_{1})-h^{-1}(s_{2})\right],
\end{equation}
showing the strict monotonicity of $h^{-1}$.

Finally, we show the continuity of $\omega$ at $\infty$. Given any sequence $z_{n}\in\mathbb{C}^{+}\cup\mathbb{R}$ with $|z_{n}|\rightarrow+\infty$,
assume, in order to derive a contradiction, that there is a subsequence
$z_{n(k)}$ such that $\sup_{k\geq1}\left|\omega(z_{n(k)})\right|<+\infty$.
Passing further to a convergent subsequence, we may assume that $\omega(z_{n(k)})$
tends to a complex number $w\in\overline{\Omega}$ as $k\rightarrow\infty$.
By the continuity of $H$ on $\overline{\Omega}$ and global inversion,
we reach at the following contradiction:
\[
\left|H(w)\right|=\lim_{k\rightarrow\infty}|H\left(\omega(z_{n(k)})\right)|=\lim_{k\rightarrow\infty}|z_{n(k)}|=+\infty.
\]
Thus, the map $\omega$ has to be continuous at $\infty$.
\qed

\addvspace{0.1in}

\textit{Proof} of $\mathbf{Proposition\;\;\ref{Hprop2}}$: For $\alpha \in \partial\Omega\cap\mathbb{R}$, we have $f(\alpha)=0$; that is to say, $H^{*}(\alpha)=H(\alpha)\in \mathbb{R}$ and $H$ maps the entire vertical line $\{\alpha+iy:y>0\}$ into $\mathbb{C}^{+}$. Given the nature of $\Im H$, it is easy to see that these conditions are all equivalent to the condition $g(\alpha)\leq b$.

Assume $g(\alpha)\leq b$. Applying Proposition \ref{prop2.1} (3) to the Nevanlinna form $F(z)=bz-H(z)$ (so that the measure $\rho$ is indeed the $F$-representing measure for $F$), we obtain the existence of $H^{\prime}(\alpha)$ and the formula for the value of $H(\alpha)$. Conversely, the existence of $H^{\prime}(\alpha)$ in $[0,+\infty)$ implies \[0\leq H^{\prime}(\alpha)=\lim_{y\rightarrow 0^+}\Re \frac{H(\alpha+iy)-H(\alpha)}{iy}=\lim_{y\rightarrow 0^+}\frac{\Im H(\alpha+iy)}{y},\] whence $g(\alpha)\leq b$. So, the statements (1) to (5) are all equivalent.

We next verify that the angular derivative $\omega^{\prime}(h(\alpha))$ exists and the formula $\omega^{\prime}(h(\alpha))=1/H^{\prime}(\alpha)$ holds. Denoting $\alpha+i\epsilon=\omega(\xi_\epsilon)$ and $\alpha=\omega(s)$, one has
\[H^{\prime}(\alpha)=\lim_{\epsilon\to0^+}\frac{H(\alpha+i\epsilon)-H(\alpha)}{i\epsilon}=\lim_{\xi_\epsilon\rightarrow s}\frac{\xi_\epsilon-s}{\omega(\xi_\epsilon)-\omega(s)}\] by global inversion, which, according to Lindel\"{o}f theorem, is further equal to \[\lim_{w \rightarrow_{\sphericalangle} s}\frac{w-s}{\omega(w)-\omega(s)}=\frac{1}{\omega^{\prime}(h(\alpha))}.\]

Now, under the hypothesis $H^\prime(\alpha)>0$, we shall prove that the boundary curve $\partial\Omega=\{x+if(x):x\in\mathbb{R}\}$ is tangent to $\mathbb{R}$ at the point $\alpha$. Without loss of generality, assume $\alpha=0$. It suffices to prove that the curve $x+if(x)\to0$ tangentially as $x\to0$, that is to prove $\lim_{x\to0}f(x)/x=0$. If there exists a sequence $\{x_n\}$ converging to $0$ so that
$\lim_{n\to\infty}|f(x_n)/x_n|=\infty$, then
\[\lim_{n\to\infty}\frac{h(x_n)-h(0)}{f(x_n)}=\lim_{n\to\infty}\frac{H(x_n+if(x_n))-H(0)}{x_n+if(x_n)}\cdot\frac{1+\frac{f(x_n)}{x_n}i}{\frac{f(x_n)}{x_n}}=iH'(0),\] which is impossible because this limit is not a real number. Similarly, if the set $\{f(x)/x:x\neq0\}$ has a non-zero limit point $L$ as $x\to0$, then $H'(0)(1/L+i)\in\mathbb{R}$, a contradiction. All these considerations reveal that $\lim_{x\to0}f(x)/x=0$, as desired.

Finally, under the assumption $H^\prime(\alpha)>0$ again, $H(\gamma(t))\rightarrow_{\sphericalangle}H(\alpha)$ holds if and only if $\gamma(t)\rightarrow_{\sphericalangle}\alpha$ as $t\rightarrow 1^{-}$. Indeed, the ``if'' part follows from Proposition 2.2, while the converse implication is a consequence of global inversion and an application of Proposition \ref{prop2.1} (3) to $\omega$.
\qed

\addvspace{0.1in}

\textit{Proof} of $\mathbf{Proposition\;\;\ref{Hprop3}}$: To prove the conformality of $H$ for $x\in V$, we first look at the case where the measure $\rho$ is degenerate,
say, $\rho=c\delta_{s_{0}}$ for some $c>0$ and $s_{0}\in\mathbb{R}$.
A straightforward calculation shows that the positive set $V=(s_{0}-r,s_{0}+r)$,
the function $f$ is a semicircular curve $f(x)=I_{[s_{0}-r,s_{0}+r]}(x)\sqrt{r^{2}-(x-s_{0})^{2}}$, and the region $\Omega=\left\{ z\in\mathbb{C}^{+}:|z-s_{0}|>r\right\} $, where the radius $r=\sqrt{c(1+s_{0}^{2})/b}$. One verifies directly
that the complex derivative $H^{\prime}(z)\neq0$ if $z\in\Omega$
or $z=x+if(x)$, $x\in V$.

For a nondegenerate $\rho$, we again consider such $z\in\partial\Omega\cap\mathbb{C}^{+}$.
The derivative of $H$ at $z$ satisfies
\begin{multline*}
\left|H^{\prime}(z)\right|=\left|b-\int_{\mathbb{R}}\frac{1+s^{2}}{(z-s)^{2}}\,d\rho(s)\right|\geq b-\left|\int_{\mathbb{R}}\frac{1+s^{2}}{(z-s)^{2}}\,d\rho(s)\right|\\
>b-\int_{\mathbb{R}}\frac{1+s^{2}}{\left|z-s\right|^{2}}\,d\rho(s)=\frac{\Im H(z)}{\Im z}\geq0,
\end{multline*}
because the quotient $(z-s)^{2}/|z-s|^{2}$ is not a constant for
$\rho$-almost all $s\in\mathbb{R}$. The assertion (1) follows.

For (2), if $I\subset\mathbb{R}\setminus\overline{V}$ then Proposition
\ref{Hprop2} (3) and Lemma \ref{f1} imply $\rho(I)=0$. As a result, the map $H$
defined by the integral form (\ref{eq:3.1}) extends analytically across $I$
by reflection, i.e., $H(z)=\overline{H\left(\overline{z}\right)}$
for $z\in\mathbb{C}^{-}\cup I$. This extension takes real values
and is strictly increasing on $I$ by \eqref{eq:5.3}. The extension properties
of the global inverse $\omega$ follows from the holomorphic inverse function
theorem.

For (3), assume the map $\omega$ extends analytically to an open
disk $D$ centered at $s_{0}\in\mathbb{R}$. On the open interval
$J=D\cap\mathbb{R}$, the map $\omega$ has a power series expansion
\[
\omega(s)=\sum\nolimits _{n=0}^{\infty}a_{n}(s-s_{0})^{n},
\]
which yields the power series representation for $\Re\omega(s)=h^{-1}(s)$:
\[
h^{-1}(s)=\sum\nolimits _{n=0}^{\infty}(\Re a_{n})(s-s_{0})^{n},\quad s\in J.
\]
The last power series also converges for $z\in D$, and
hence it serves as an analytic extension for $h^{-1}$.
We use the same notation $h^{-1}$ to denote this extension.

For a further reference, we note that $\Im\omega$ is also real analytic
at $s_{0}$ by the same argument, and therefore it has a complex analytic extension near $s_0$. We write $\Im\omega$
for this analytic continuation.

The value of the complex derivative of $\omega$ at $s_{0}$ is not zero because $2b|\omega^{\prime}(s_0)|\geq 1$. This implies that $h^{-1}$ has the complex derivative $(h^{-1})^{\prime}(s_{0})=\Re a_{1}\neq0$.
Therefore, the inverse function theorem implies that $h^{-1}$ has
an analytic inverse $\widetilde{h}$ defined near the point $x_{0}=h^{-1}(s_{0})$.
The map $\widetilde{h}$ then serves as an analytic continuation for the homeomorphism
$h$, proving that $h$ is real analytic at $x_{0}$.

For the function $f$, in view of the inversion relationship
\begin{equation} \label{eq:5.4}
x+if(x)=\omega\left(H(x+if(x))\right)=\omega\left(h(x)\right),
\end{equation}
we take the composition $(\Im\omega)\circ\widetilde{h}$ of complex analytic
functions as the analytic continuation of $f$ near $x_{0}$.

To prove the converse in (3), assume that both $h$ and $f$ extend real analytically
(and hence complex analytically) to $x_{0}$. The complex derivative
$h^{\prime}(x_{0})$ coincides with its real derivative, which
is strictly positive because $h$ is strictly increasing on $\mathbb{R}$.
By the inverse function theorem again, $h^{-1}$ extends analytically
to $s_{0}=h(x_{0})$. The analyticity of $\omega$ at $s_{0}$ now
follows from the inversion relationship \eqref{eq:5.4}.

Finally, the assertion (4) is a direct application of the inverse function theorem. Note that the angular derivative of $\omega$ at any $s\in \mathbb{R}$ is always non-zero, because $\omega$ is an analytic self-map of $\mathbb{C}^{+}$.
\qed

\addvspace{0.1in}

\textit{Proof} of $\mathbf{Theorem\;\;\ref{boundaryR}}$:
The sets $A$, $B$, and $C$ are mutually disjoint because
they are distinguished by different values of the vertical limit
$G_{\nu}^{*}(\alpha)$. We shall prove that they form a partition
of the zero set $\partial\Omega\cap\mathbb{R}$.

By translating the measure $d\nu$ to $d\nu(s+\alpha)$, we may and do assume that $\alpha=0$. We first assume $0\in\partial\Omega\cap\mathbb{R}$
and prove that $0\in A\cup B\cup C$. Then Proposition \ref{Hprop2} says that $\Im H(iy)>0$
for all $y>0$. Since $H(z)=z+\varphi_{\mu}\left(F_{\nu}(z)\right)$,
we obtain
\begin{equation}\label{eq:5.5}
\frac{\Im F_{\nu}(iy)}{y}\int_{\mathbb{R}}\frac{(1+s^{2})\,d\sigma_{\mu}(s)}{|F_{\nu}(iy)-s|^{2}}=\frac{-\Im G_{\nu}(iy)}{y}\int_{\mathbb{R}}\frac{(1+s^{2})\,d\sigma_{\mu}(s)}{|1-sG_{\nu}(iy)|^{2}}<1
\end{equation} for all $y>0$.
We also know from the monotone convergence theorem that
\[
I_{1}=\int_{\mathbb{R}}\frac{d\nu(s)}{s^{2}}=\lim_{y\rightarrow0^{+}}\frac{-\Im G_{\nu}(iy)}{y}\in(0,+\infty]
\]
and
\[
I_{2}=1+\int_{\mathbb{R}}\frac{1+s^{2}}{s^{2}}\,d\rho_{\nu}(s)=\lim_{y\rightarrow0^{+}}\frac{\Im F_{\nu}(iy)}{y}\in(1,+\infty].
\]

There are two possibilities to consider.

Case (I): $I_{1}=+\infty$. The estimate \eqref{eq:5.5} implies
\begin{equation}\label{eq:5.6}
\lim_{y\rightarrow0^{+}}\int_{\mathbb{R}}\frac{1+s^{2}}{|1-sG_{\nu}(iy)|^{2}}\,d\sigma_{\mu}(s)=0.
\end{equation}
It follows that the vertical limit $F_{\nu}^{*}(0)=\lim_{y\rightarrow0^{+}}F_{\nu}(iy)=0$.
Indeed, if this were not true then we would be able to find a sequence $y_{n}\downarrow0$ so that $\sup_{n\geq1}|G_{\nu}(iy_{n})|<+\infty$.
By dropping to a convergent subsequence if necessary, we further assume
that $G_{\nu}(iy_{n})\rightarrow z\in\mathbb{C}$. Then the limit
\eqref{eq:5.6} and Fatou's lemma yield
\[
0\geq \int_{\mathbb{R}}\frac{1+s^{2}}{|1-sz|^{2}}\,d\sigma_{\mu}(s)>0,
\]
a contradiction. Now, with the fact $F_{\nu}^{*}(0)=0$, we apply
Fatou's lemma to \eqref{eq:5.5} again to conclude that $I_{2}<+\infty$,
\[
I_{3}=\int_{\mathbb{R}}\frac{1+s^{2}}{s^{2}}\,d\sigma_{\mu}(s)<+\infty,
\]
and $I_{2}I_{3}\leq1$. This shows that $0\in A$.

Case (II): $I_{1}<+\infty$. Proposition \ref{prop2.1} (2) implies that the
vertical limit $G_{\nu}^{*}(0)$ exists and belongs to $\mathbb{R}$.
Taking the limit infinimum as $y\rightarrow0^{+}$ in \eqref{eq:5.5}, we get
\[
I_{4}=\int_{\mathbb{R}}\frac{1+s^{2}}{(1-sG_{\nu}^{*}(0))^{2}}\,d\sigma_{\mu}(s)\leq\frac{1}{I_{1}},
\]
showing that $0\in B$ if $G_{\nu}^{*}(0)\in\mathbb{R}\setminus\left\{ 0\right\} $
and that $0\in C$ if $G_{\nu}^{*}(0)=0$.

We have shown that $(\partial\Omega\cap\mathbb{R})\subset A\cup B\cup C$.
The proof of the opposite inclusion relies on the fact:

\begin{eqnarray*}
g(0) & = & \lim_{y\rightarrow0^{+}}\frac{-\Im H(iy)+y}{y}\\
 & = & \lim_{y\rightarrow0^{+}}\frac{\Im F_{\nu}(iy)}{y}\int_{\mathbb{R}}\frac{1+s^{2}}{|F_{\nu}(iy)-s|^{2}}\,d\sigma_{\mu}(s)\\
 & = & \lim_{y\rightarrow0^{+}}\frac{-\Im G_{\nu}(iy)}{y}\int_{\mathbb{R}}\frac{1+s^{2}}{|1-sG_{\nu}(iy)|^{2}}\,d\sigma_{\mu}(s).
\end{eqnarray*}
Thus, if $0\in A$ then we have $F_{\nu}^{*}(0)=0$ and $I_{2},I_{3}<+\infty$
with $I_{2}I_{3}\leq1$. Moreover, Proposition 2.1 (3) shows that $F_{\nu}(iy)\rightarrow_{\sphericalangle}0$ as $y\rightarrow 0^+$. Thus the dominated convergence theorem with the dominating function
\[\frac{1+s^2}{|\xi-s|^2}=\left|\frac{\xi}{\xi-s}-1\right|^2\frac{1+s^2}{s^2}\leq(\sqrt{(\Re\xi/\Im\xi)^2+1}+1)^2\frac{1+s^2}{s^2}\in L^2(\sigma_\mu)\] for any $\mathbb{C}^+\ni\xi\rightarrow_{\sphericalangle}0$ implies $g(0)=I_{2}I_{3}\leq1$. If $0\in B\cup C$, then we have
$G_{\nu}^{*}(0)\in\mathbb{R}$ and $I_{1},I_{4}<+\infty$ with $I_{1}I_{4}\leq1$,
implying further that $g(0)=I_{1}I_{4}\leq1$. In all cases we have verified that $g(0) \leq 1$, and this means $0\in\partial\Omega\cap\mathbb{R}$ by Proposition 3.2.

In conclusion, we have $\partial\Omega\cap\mathbb{R}= A\cup B\cup C$.
Additionally, our arguments and Proposition \ref{Hprop2} indicate that the angular
derivative
\[
\omega^{\prime}(h(\alpha))=\frac{1}{H^{\prime}(\alpha)}=\frac{1}{1-g(\alpha)}=\begin{cases}
1/(1-I_{2}I_{3}) & \text{if }\alpha\in A, \\
1/(1-I_{1}I_{4}) & \text{if }\alpha\in B\cup C.
\end{cases}
\]\qed

\addvspace{0.1in}

\textit{Proof} of $\mathbf{Theorem\;\;\ref{thmR1}}$: Throughout the proof we use the two boundary parametrizations
$s=h(x)=H(x+if(x))$ and $\omega(s)=x+if(x)$ for $x\in\mathbb{R}$.
As usual, the indicator $I_\emptyset$ of the empty set $\emptyset$ is set
to be the constant zero function.

We have seen in Example \ref{exam3.8} that the finite set $D=h(A)$ supports
the singular (atomic) part of $\mu\boxplus\nu$. So the absolutely
continuous part $(\mu\boxplus\nu)_{\text{ac}}$ is supported on $\mathbb{R}\setminus D$,
written as a disjoint union $\mathbb{R}\setminus D=h(V)\cup h(B\cup C)$,
and
\begin{eqnarray*}
d(\mu\boxplus\nu)_{\text{ac}}(s) & = & \pi^{-1}\left(-\Im G_{\mu\boxplus\nu}\right)^{*}(s)\,d\lambda(s)\\
 & = & \pi^{-1}\left[I_{h(V)}(s)+I_{h(B\cup C)}(s)\right]\left(-\Im G_{\mu\boxplus\nu}\right)^{*}(s)\,d\lambda(s).
\end{eqnarray*}
In particular, we have the push-forward formula:
\begin{equation} \label{eq:5.7}
d(\mu\boxplus\nu)_{\text{ac}}(h(x))=\pi^{-1}\left[I_{V}(x)+I_{B\cup C}(x)\right]\left(-\Im G_{\mu\boxplus\nu}\right)^{*}(h(x))\,d\lambda(h(x)).
\end{equation}

We shall identify the limit $\left(-\Im G_{\mu\boxplus\nu}\right)^{*}(h(x))$
separately on each part. We first address the case of $x\in V$. Observe
that the composition $G_{\nu}\circ\omega$ extends continuously to a map from $\mathbb{C}^{+}\cup h(V)$ to $\mathbb{C}^{-}$, because $\omega(s)\in\mathbb{C}^{+}$. So
the subordination $G_{\mu\boxplus\nu}=G_{\nu}\circ\omega$ in $\mathbb{C}^{+}$
implies that \begin{multline*}
\left(-\Im G_{\mu\boxplus\nu}\right)^{*}(h(x))=(-\Im G_{\nu}\circ\omega)^{*}(h(x))=(-\Im G_{\nu}\circ\omega)(h(x))\\
=-\Im G_{\nu}(\omega(s))=-\Im G_{\nu}(x+if(x))=f(x)\int_{\mathbb{R}}\frac{d\nu(s)}{(x-s)^{2}+f(x)^{2}}.
\end{multline*} As for points $x$ in $B\cup C$, the continuity of $F_{\mu\boxplus\nu}$ from [\ref{BBG09}] and Theorem \ref{boundaryR} show that $\left(-\Im G_{\mu\boxplus\nu}\right)^{*}(h(x))=0$.

By the preceding discussion, we now define the function $p_{\mu\boxplus\nu}:\mathbb{R}\rightarrow[0,\infty)$
by
\[
p_{\mu\boxplus\nu}(h(x))=\frac{f(x)}{\pi}\int_{\mathbb{R}}\frac{d\nu(s)}{(x-s)^{2}+f(x)^{2}},\qquad x\in V,
\]
and $p_{\mu\boxplus\nu}\circ h=0$ on $\mathbb{R}\setminus V$, so
that the formula \eqref{eq:5.7} can be re-casted into
\[
d(\mu\boxplus\nu)_{\text{ac}}=p_{\mu\boxplus\nu}\,d\lambda.
\] The continuity of this version of density follows from that of $F_{\mu\boxplus\nu}$.

We next address the local analyticity of $p_{\mu\boxplus\nu}$. The analyticity on $\mathbb{R}\setminus \overline{V}$ is somehow trivial since any $s \in h(\mathbb{R}\setminus \overline{V})$ belongs to an open interval on which the map $p_{\mu\boxplus\nu}$ is constantly zero. The analyticity of $p_{\mu\boxplus\nu}$ on $h\left(V\right)$ follows directly from Proposition \ref{Hprop3} (1) and (3) and the integral representation of $p_{\mu\boxplus\nu}$. In conclusion, $p_{\mu\boxplus\nu}$ is the version of the Radon-Nikodym derivative $d(\mu\boxplus\nu)_{\text{ac}}/d\lambda$ that we are looking for, that is, the assertion (2) is now proved.

The continuity of $p_{\mu\boxplus\nu}$ implies that the topological
support
\[
\text{supp}((\mu\boxplus\nu)_{\text{ac}})=\text{supp}(p_{\mu\boxplus\nu})=\overline{h(V)}.
\]

To finish the proof of the assertion (1), we need to show that $\text{supp}(\nu_{\text{sc}})\cup\text{supp}(\nu_{\text{ac}})\subset\overline{V}$.
Suppose, in order to derive a contradiction, that there exists $x\in[\text{supp}(\nu_{\text{sc}})\cup\text{supp}(\nu_{\text{ac}})]\cap[\mathbb{R}\setminus\overline{V}]$. It follows that there exists a closed interval $J$ such that $x\in J$,
$J\subset B\cup C$, and $\nu(J)>0$. In this case it is elementary
to find nested closed subintervals $J\supset I_{1}\supset I_{2}\supset\cdots$
satisfying $\lambda(I_{n})=\lambda(J)2^{-n}$, $\nu(I_{n})\geq\nu(J)2^{-n}$,
and $\bigcap_{n=1}^{\infty}I_{n}=\{\alpha\}$. Then we reach at
\[
\infty>\int_{\mathbb{R}}\frac{d\nu(s)}{(\alpha-s)^{2}}\geq\frac{\nu(J)}{\lambda(J)^{2}}\,2^{n},\quad n\geq1,
\]
a contradiction. The assertion (1) is therefore proved.

The assertion (3) follows from writing the open set $V$ as a disjoint
union of at most countably many open intervals.

Finally, we show the assertion (4). If $A$ is not empty and $\alpha\in A$, then the inequalities in Theorem \ref{boundaryR} and Example \ref{exampleR1} imply that the point $s_{\mu}$ exists. Moreover, as seen in the proof of Theorem \ref{boundaryR}, we have $1/\omega^{\prime}(h(\alpha))=1-I_{2}I_{3}$. By Corollary \ref{measureatom} and Example \ref{exampleR1}, the last identity can be re-written as
\[
\frac{\nu(\left\{ \alpha\right\} )}{\omega^{\prime}(h(\alpha))}=\nu(\left\{ \alpha\right\} )+\mu(\{s_{\mu}\})=\mu\boxplus\nu(\{\alpha+s_{\mu}\}).
\] The last equality follows from the results in [\ref{BerVoicu98}].
The proof is finished by observing that $\alpha+s_{\mu}=h(\alpha)$
from Proposition \ref{Hprop2} and Example \ref{exampleR1}.
\qed

\addvspace{0.1in}

\textit{Proof} of $\mathbf{Theorem\;\;\ref{nocomponentR2}}$: In view of Theorem \ref{thmR1}, it suffices to show that $n\left(\overline{V}\right)<\infty$, where $V=\left\{ x\in\mathbb{R}:g(x)>1\right\}$ and
\begin{equation} \label{eq:5.8}
g(x)=\int_{\mathbb{R}}\frac{1+s^{2}}{(x-s)^{2}}\, d\rho(s),\qquad x\in\mathbb{R}.
\end{equation}

We assume $V\neq (-\infty,\infty)$ to avoid the triviality. By writing the open set $V$ into a countable union of components, its closure $\overline{V}$
can be written as a countable union of disjoint closed intervals as
follows:
\[
\overline{V}=\bigcup_{k\in K}[a_{k},b_{k}]\cup(-\infty,a]\cup[b,\infty),
\]
where $-\infty<a_{k}<b_{k}<\infty$ for all $k\in K$, $-\infty\leq a<b\leq\infty$,
the index set $K$ is at most countable, and the intervals $(-\infty,-\infty]$
and $[\infty,\infty)$ are interpreted as the empty set.
We also assume that the collection
\[
\mathcal{I}=\left\{ [a_{k},b_{k}]:k\in K\right\}
\]
of bounded components in $\overline{V}$ is not empty and aim to show that it is a finite set, with an estimate for the cardinality of $\mathcal{I}$. We do this by partitioning $\mathcal{I}$ into \[\mathcal{I}=\{I\in \mathcal{I}:I\cap\text{supp}(\nu)\neq\emptyset\}\cup \{I\in \mathcal{I}:I\cap\text{supp}(\nu)=\emptyset\}\] and then count the cardinality of the two sets in this partition.

Case (1): $I\cap\text{supp}(\nu)\neq\emptyset$. For such an $I$, there exists a component $J_{I}$ in $\text{supp}(\nu)$ such that $I\cap J_{I} \neq\emptyset$. If $J_I$ is a singleton set then $J_I \subset I$. When $J_I$ is an interval, we claim that we still have $J_I \subset I$. Indeed, assume without loss of generality that $I=[a_k,b_k]$, $J_I=[c,d]$, and $-\infty\leq c\leq b_k<d\leq \infty$. Then the maximality of $I$ implies that $(b_k,b_k+\epsilon)\subset \text{supp}(\nu) \cap (\mathbb{R}\setminus \overline{V})$ for all small $\epsilon>0$. Since $\text{supp}(\nu_{\text{sc}})\cup \text{supp}(\nu_{\text{ac}})\subset  \overline{V}$, this means that there will be uncountably many atoms of $\nu$ lying in $\mathbb{R}\setminus \overline{V}$, a contradiction. Therefore we must have $J_{I}\subset I$. Hence we can conclude that \[\# \{I\in \mathcal{I}:I\cap\text{supp}(\nu)\neq\emptyset\} \leq n(\text{supp}(\nu)).\]

Case (2): $I\cap\text{supp}(\nu)=\emptyset$. Every such an $I=[a_k,b_k]$ is contained in a
unique component $I_1$ of the open set $\mathbb{R}\setminus\text{supp}(\nu)$. Notice that $G_{\nu}$ continues analytically
across $I_{1}$ by reflection and $G_{\nu}$ is strictly decreasing on $I_{1}$. Our argument now divides into two
mutually exclusive subcases according to whether $G_{\nu}$ vanishes on the interval $I_{1}$ or not. Both subcases would show that each given component $I_1$ can only contain finitely many such $I$.

Subcase (2.1): $G_{\nu}(x)\neq0$ for all $x\in I_{1}$. In this case
we show that the closed interval $J=F_{\nu}(I)$ must contain
a component $C_{J}$ in $\text{supp}(\sigma_{\mu})$.

We first observe that $J\cap\text{supp}(\sigma_{\mu})\neq\emptyset$. Indeed,
if $J$ is disjoint from this support, then the function
\begin{equation} \label{eq:5.9}
H(z)=z+(\varphi_{\mu}\circ F_{\nu})(z)=z+\gamma_{\mu}+\int_{\mathbb{R}}\frac{1+F_{\nu}(z)s}{F_{\nu}(z)-s}\, d\sigma_{\mu}(s)
\end{equation}
extends analytically across the interval $I$ by reflection. Since $H$ admits the integral form (\ref{eq:3.1}) with the measure $\rho$ in \eqref{eq:5.8}, an application of Proposition \ref{prop2.4} (1) to the function $F(z)=z-H(z)$ implies that the measure $\rho$ does not charge
the interval $I$. Because $a_k,b_k \notin V$, we have $g(a_k),g(b_k)\leq 1$ by Proposition \ref{Hprop2}. Since $g$
is $C^2$ and strictly convex on $(a_k,b_k)$, it follows that $g\leq 1$ on the whole interval $(a_k,b_k)$. Proposition \ref{Hprop2} then implies that $I\subset \mathbb{R}\setminus V$, which contradicts to the fact $I\subset \overline{V}$. Thus, $J$
must intersect the support of $\sigma_{\mu}$.

As a result of $J\cap\text{supp}(\sigma_{\mu})\neq\emptyset$, there
will be a closed component $C_{J}$ in $\text{supp}(\sigma_{\mu})$ such
that $C_{J}\cap J\neq\emptyset$. We claim that $C_{J}\subset J$ in this case. Suppose,
in order to get a contradiction, that $C_{J}\cap(\mathbb{R}\setminus J)\neq\emptyset$.
Without loss of generality, we assume $J=[A,B]$ and $C_{J}=[C,D]$
with $-\infty \leq C\leq B<D\leq \infty$. Choose a small $\epsilon>0$ so that $(B,B+\epsilon)\subset C_{J}\setminus J$
and that the pre-image $F_{\nu}^{-1}\left((B,B+\epsilon)\right)$ under the increasing map $F_{\nu}$ is an open interval $I_2=(b_{k},c_{k})$ adjacent to $I=[a_{k},b_{k}]$ and $I_2 \subset I_{1}$. Notice that $c_k\in\mathbb{R}\setminus\overline{V}$,
and so we may assume $I_2\subset\mathbb{R}\setminus\overline{V}$ by selecting sufficiently small $\epsilon$ if needed.
Since $g\leq1$ on $I_2$, we have $\rho(I_2)=0$ by Lemma \ref{f1}. This
and \eqref{eq:5.9} imply that $\sigma_{\mu}(F_{\nu}(I_2))=0$, contradicting
to the fact that $F_{\nu}(I_2)\subset C_{J}\subset\text{supp}(\sigma_{\mu})$.
Therefore, we must have $C_{J}\subset J$ in this case.

We now count the number of components $I$ in Subcase (2.1) as follows. First, to each fixed $I_1$, if two different $I,I^\prime \in \mathcal{I}$ are contained in $I_1$ then the strict monotonicity of $F_{\nu}$ on $I_1$ shows that $J=F_{\nu}(I)$ and $J^{\prime}=F_{\nu}(I^{\prime})$ are disjoint and hence $C_J \neq C_{J^{\prime}}$. Therefore $I_1$ can contain at most $n(\text{supp}(\sigma_{\mu}))$ many $I$'s from $\mathcal{I}$. Secondly, there are at most $n(\mathbb{R}\setminus\text{supp}(\nu))$ many $I_1$'s from the assumption $n(\text{supp}(\nu))<\infty$. We conclude that \[\# \{I \,\,\text{in Subcase (2.1)}\} \leq n(\text{supp}(\sigma_{\mu}))\,n(\mathbb{R}\setminus\text{supp}(\nu)).\]

Subcase (2.2): $G_{\nu}(x_{1})=0$ for some $x_{1}\in I_{1}$. Then
such a zero $x_{1}$ is unique due to the strict monotonicity of $G_{\nu}$ on $I$. We partition this subcase into \[\{I \,\,\text{in Subcase (2.2)}\}=\{I:x_1\in I\}\cup \{I:x_1<\min{I}\}\cup\{I:\max{I}<x_1\}.\]

To each given $I_1$ in this subcase, there can only be one $I$ satisfying $x_1\in I$. So we have $\# \{I:x_1\in I\}\leq n(\mathbb{R}\setminus\text{supp}(\nu))$.

For the case $x_1<\min{I}$, we are counting the number of $I$'s lying on the right side of the zero $x_{1}$. By replacing $I_1=(a,b)$ with the smaller interval $(x_1,b)$, the counting simply reduces to that in Subcase (2.1) because $G_{\nu}$ vanishes only at $x_{1}$. (Note that the argument in Subcase (2.1) does not depend on the maximal connectedness of $I_1$.) So we have $\# \{I:x_1<\min{I}\}\leq n(\text{supp}(\sigma_{\mu}))\,n(\mathbb{R}\setminus\text{supp}(\nu))$.

The remaining case of $\max{I}<x_1$ is clearly about those $I$'s on the left side of the unique zero $x_{1}$. We have $\# \{I:\max{I}<x_1\}\leq n(\text{supp}(\sigma_{\mu}))\,n(\mathbb{R}\setminus\text{supp}(\nu))$, finishing the counting in Subcase (2.2).

In conclusion, we have shown that $\mathcal {I}$ is a finite set and \begin{multline*}n\left(\text{supp}((\mu\boxplus\nu)_{\text{ac}})\right)=n\left(h(\overline{V})\right)=n(\overline{V})\\ \leq 2+n(\text{supp}(\nu))+[1+3n(\text{supp}(\sigma_{\mu}))]n(\mathbb{R}\setminus\text{supp}(\nu)).\end{multline*} Theorem \ref{thmR1} and Example \ref{exam3.8} imply further that \begin{multline*}n\left(\text{supp}(\mu\boxplus\nu)\right)\leq n\left(\text{supp}((\mu\boxplus\nu)_{\text{ac}})\right)+n\left(\text{supp}((\mu\boxplus\nu)_{\text{s}})\right) \\ \leq  n\left(\text{supp}((\mu\boxplus\nu)_{\text{ac}})\right)+ \text{Cardinality}(A).\end{multline*}\qed

It was pointed out to us by an anonymous referee that our estimate may not be sharp, since there might be some atoms of $\mu\boxplus\nu$ inside the connected components of the absolutely continuous support $\text{supp}((\mu\boxplus\nu)_{\text{ac}})$.

\subsection{Property (H)} \indent\par \textit{Proof} of $\mathbf{Theorem\;\;\ref{propertyH1}}$: We have seen from the monotone convergence theorem that \begin{equation}\label{eq:5.10}-\lim_{\epsilon\rightarrow 0^{+}}\epsilon^{-1}\Im\varphi_\mu(x+i\epsilon) =\int_\mathbb{R}\frac{1+s^2}{(x-s)^2}\,d\sigma_\mu(s),\quad x\in \mathbb{R},\end{equation} and $-y\Im\varphi_\mu(iy)\rightarrow \sigma_\mu(\mathbb{R})+m_2(\sigma_\mu)$ as $y\rightarrow \infty$.

The case of a degenerate $\nu$ reduces to Proposition \ref{anadensityR}. We shall
assume that $\nu$ is nondegenerate. By Theorem \ref{boundaryR} and Remark \ref{remarkR2}, the value of
\eqref{eq:5.10} being strictly greater than $1$ means that there are no type $A$ or type $B$
boundary points for $\mu\boxplus\nu$, and the condition $\sigma_{\mu}(\mathbb{R})+m_{2}(\sigma_{\mu})=+\infty$
implies that the type $C$ points do not exist. In other words, we have $V=\mathbb{R}$ and $p_{\mu\boxplus\nu}>0$
everywhere on $\mathbb{R}$.

Conversely, we assume $V=\mathbb{R}$ for all $\nu\in\mathcal{P}_{\mathbb{R}}$. When $\nu=\delta_0$, Proposition \ref{anadensityR} shows that the integral in \eqref{eq:5.10} exceeds 1. Assume, in order to obtain a contradiction, that $var(\mu)<+\infty$. Consider $\nu\in \mathcal{P}_{\mathbb{R}}$ such that $F_\nu(z)=z+(1+z)(1-z)^{-1}var(\mu)$, that is, $a=0$ and $\rho_\nu=var(\mu)\,\delta_1$ in (2.1). Remark \ref{remarkR2} shows that the point $\alpha=1$ belongs to the set $C$ associated with such a $\nu$, a contradiction to $V=\mathbb{R}$. So we must have $\sigma_\mu(\mathbb{R})+m_2(\sigma_\mu)=var(\mu)=+\infty$.\qed

\addvspace{0.1in}

\subsection{Analyticity at zeros} \indent\par \textit{Proof} of $\mathbf{Proposition\;\;\ref{nonanaR}}$: For the
assertion (1), suppose that $p_{\mu\boxplus\nu}$ is analytic on the open interval $I=(s_0-\epsilon,s_0+\epsilon)$. We have seen that such a function $p_{\mu\boxplus\nu}$ admits an analytic continuation (still denoted by $p_{\mu\boxplus\nu}$) to $D=\{z\in
\mathbb{C}:|z-s_0|<\epsilon\}$ through its power series expansion. If the zero $s_0$ is not isolated, then $p_{\mu\boxplus\nu}$ would be constantly zero on $I$, contradicting the fact $s_0\in \text{supp}((\mu\boxplus\nu)_{\text{ac}})$.

The case when $\alpha$ is not an isolated point in $\partial\Omega\cap\mathbb{R}$ has been dealt with in (1). In the assertion (2), Now let $\alpha$ be an isolated point in $\partial\Omega\cap\mathbb{R}$. In (2), if $\alpha\in B \cup C$, then the analyticity of $p_{\mu\boxplus\nu}$ at $s_0$ and Julia-Carath\'{e}odory theory
yield $G_{\mu\boxplus\nu}^\prime(s_0)\in \mathbb{R}$, which leads to the following contradiction
\begin{align*}
-G_{\mu\boxplus\nu}^\prime&(s_0)=\liminf_{z\to s_0}\frac{-\Im G_{\mu\boxplus\nu}(z)}{\Im
z}=\liminf_{z\to s_0}\cdot\frac{-\Im
G_\nu(\omega(z))}{\Im\omega(z)}\frac{\Im\omega(z)}{\Im z} \\
&\geq \left(\liminf_{z\to\alpha}\frac{-\Im G_\nu(z)}{\Im
z}\right)\left(\liminf_{z\to s_0}\frac{\Im\omega(z)}{\Im
z}\right)=\int_{\mathbb{R}} \frac{d\nu(s)}{(\alpha-s)^2}\cdot\omega^{\prime}(s_0)=+\infty
\end{align*} by Proposition \ref{prop2.1}. The case of $\alpha\in A$ is proved in the same way by making use of Corollary \ref{measureatom}. \qed

\addvspace{0.1in}

\textit{Proof} of $\mathbf{Theorem\;\;\ref{analyticR1}}$: We first remark that our hypothesis $\omega^{\prime}(s_{0})<+\infty$
and Proposition \ref{prop2.1} (3) imply that $H^{\prime}(\alpha)=1/\omega^{\prime}(s_{0})\neq0$
and $\omega(z)\rightarrow_{\sphericalangle}\alpha$ as $z\rightarrow_{\sphericalangle}s_{0}$.
Also, the analyticity of the measure $\sigma_{\mu}$ at the point
$F_{\nu}^{*}(\alpha)$ is equivalent to that of the integral form
$\varphi_{\mu}$ at the same point by Proposition \ref{prop2.4} and Corollary \ref{cor2.6}. In the case $\alpha\in C$,
this means that $R_{\mu}(z)=\varphi_{\mu}(1/z)$ is analytic at $z=0$.

We first consider the case $\alpha\in B$. Assume $p_{\mu\boxplus\nu}$ is analytic
at $s_{0}$. Proposition \ref{prop2.4} implies that $G_{\mu\boxplus\nu}$ extends
analytically to the point $s_{0}$, with the value
\[
G_{\mu\boxplus\nu}(s_{0})=\lim_{y\rightarrow0^{+}}G_{\nu}\left(\omega(s_{0}+iy)\right)=G_{\nu}^{*}(\alpha)\in\mathbb{R}\setminus\left\{ 0\right\} .
\]
Hence the transform $F_{\mu\boxplus\nu}$ extends analytically
to the point $s_{0}$ and $F_{\mu\boxplus\nu}(s_{0})=F_{\nu}^{*}(\alpha)$.
The inverse relationship between $H$ and $\omega$ yields
\[
z=H(\omega(z))=\omega(z)+(\varphi_{\mu}\circ F_{\mu\boxplus\nu})(z),\qquad z\in\mathbb{C}^{+}.
\]
Combining with the analyticity of $\varphi_{\mu}$ at $F_{\nu}^{*}(\alpha)=F_{\mu\boxplus\nu}(s_{0})$, we see that the
map $\omega$ extends analytically to the point $s_{0}$ and the value
of $\alpha=\omega(s_{0})$ is given by
\[
\alpha=s_{0}-\varphi_{\mu}\left(F_{\nu}^{*}(\alpha)\right)=s_{0}-\gamma_{\mu}+\int_{\mathbb{R}}\frac{1+sF_{\nu}^{*}(\alpha)}{s-F_{\nu}^{*}(\alpha)}\,d\sigma_{\mu}(s),
\]
where the convergence of the integral is ensured by the fact $\alpha\in B$ and Proposition \ref{prop2.1} (3). Then Proposition \ref{Hprop3} (4) shows that $H$ extends analytically
to $\alpha$, which allows us to conclude that $G_{\nu}$
also extends analytically to $\alpha$ because $G_{\nu}=G_{\mu\boxplus\nu}\circ H$.
Thus, the analyticity of the measure $\nu$ follows from Proposition \ref{prop2.4}.

Conversely, assume that $G_{\nu}$ extends analytically to $\alpha$.
Since $\alpha\in B$, the reciprocal $F_{\nu}$ extends analytically
to $\alpha$ and $F_{\nu}(\alpha)=F_{\nu}^{*}(\alpha)$.
In view of the definition $H(z)=z+(\varphi_{\mu}\circ F_{\nu})(z)$ for $z\in\mathbb{C}^{+}$,
the map $H$ extends analytically to the point $\alpha$, and so does
the map $\omega$ to the point $s_0$ by Proposition \ref{Hprop3} (4). The analyticity of the free
convolution $\mu\boxplus\nu$ now follows from the subordination $G_{\mu\boxplus\nu}=G_{\nu}\circ\omega$ and Proposition \ref{prop2.4}.

The case of $\alpha\in C$ is proved in the same way, except one uses the equations
\[H(z)=z+(R_{\mu}\circ G_{\nu})(z),\quad
z=\omega(z)+(R_{\mu}\circ G_{\mu\boxplus\nu})(z)\]
to show the analytic continuation of $H$ and $\omega$.

In the case of $\alpha\in A$, having a strict inequality in Theorem \ref{boundaryR} gives that
$\mu\boxplus\nu(\{s_{0}\})>0$ by Remark \ref{remarkR2} and that $\nu(\{\alpha\})>0$ and
the angular derivatives
\begin{equation}\label{eq:5.12}
F_{\mu\boxplus\nu}^{\prime}(s_{0})=F_{\nu}^{\prime}(\alpha)\omega^{\prime}(s_{0})
\end{equation}
by Theorem \ref{thmR1} (4). Note that $s_0$ is an isolated atom of $\mu\boxplus\nu$ since $\mu$ has only one atom. As we have seen in the case $\alpha\in B$, the analytic continuation of $F_{\mu\boxplus\nu}$ at the point $s_0$ amounts to that of $F_\nu$ at $\alpha$.
Moreover, in view of \eqref{eq:5.12}
the point $s_{0}$ is a simple zero for $F_{\mu\boxplus\nu}$ if and
only if $\alpha$ is a simple zero for $F_{\nu}$.
By Corollary \ref{analyextR2}, the density $p_{\mu\boxplus\nu}$ is analytic at
$s_{0}$ if and only if $F_{\mu\boxplus\nu}$
extends analytically to the point $s_{0}$ and this extension has
$s_{0}$ as a simple zero. Therefore, the analyticity of $p_{\mu\boxplus\nu}$ at $s_0$ yields that $\alpha$ is an isolated atom of $\nu$ for otherwise $F_\nu\equiv0$ and that the singular continuous part of the restriction of $\nu$ on some open interval containing $\alpha$ vanishes. Consequently, by Corollary
\ref{analyextR2} again we conclude that $p_{\mu\boxplus\nu}$ is analytic at $s_{0}$ if and only if the measure $\nu$ is meromorphic at
$\alpha$. \qed

\section*{Acknowledgments} The authors would like to thank an anonymous referee for valuable comments on this work and pointing out one mistake in the proof of Theorem 3.19. The first author
was supported by a grant from the Ministry of Science and Technology in Taiwan MOST 110-2628-M-110-002-MY4. The second author was supported by the
NSERC Canada Discovery Grant RGPIN-2016-03796.

\end{document}